\documentclass[10pt]{article}

\usepackage{amscd,amsmath}
\usepackage{amsfonts}
\usepackage{amssymb,amsthm}
\usepackage{longtable}
\usepackage{graphicx}
\usepackage[all]{xy}
\usepackage{hyperref}
\usepackage{enumitem}
\usepackage{a4wide}
\usepackage{sseq}

\usepackage{caption}
\usepackage{subcaption}

\usepackage{url}

\theoremstyle{definition}
\newtheorem{Thm}{Theorem}[section]
\newtheorem{Def}{Definition}[section]
\newtheorem{Lem}[Thm]{Lemma}
\newtheorem{Cor}[Thm]{Corollary}
\newtheorem{prop}[Thm]{Proposition}

\newtheorem{Rem}{Remark}[section]

\newtheorem{Ex}{Example}[subsection]

\newtheorem{theorem}{Theorem}

\newcommand{\ow}{\omega}

\newcommand{\lda}{\lambda}

\newcommand{\C}{{\mathbb{C}}}
\newcommand{\R}{{\mathbb{R}}}

\newcommand{\Z}{{\mathbb{Z}}}
\newcommand{\N}{{\mathbb{N}}}

\DeclareMathOperator{\Sp}{Sp}

\DeclareMathOperator{\Diff}{Diff}

\DeclareMathOperator{\id}{\text{id}}
\DeclareMathOperator{\im}{Im}

\DeclareMathOperator{\crit}{crit}\DeclareMathOperator{\shift}{shift}
\DeclareMathOperator{\Fix}{\text{Fix}}
\DeclareMathOperator{\lcm}{lcm}
\DeclareMathOperator{\symp}{\text{Symp}}
\DeclareMathOperator{\WFH}{WFH}
\DeclareMathOperator{\LFH}{LFH}
\DeclareMathOperator{\WFC}{WFC}
\DeclareMathOperator{\Spec}{\text{Spec}}
\DeclareMathOperator{\ham}{\text{Ham}}

\newcommand{\p}{\partial}

\begin{document}

%%%%%%%%%%%%%%%%%%%%%%%%%%%%%%%%%%%%%%%%%%%%%%%%%%%%%%%%%%%%%%%%%
%%%%%%%%%%%%%%%%%%%%%%% Title
%%%%%%%%%%%%%%%%%%%%%%%%%%%%%%%%%%%%%%%%%%%%%%%%%%%%%%%%%%%%%%%%%

\title{Volume growth in the component of  fibered twists}

\author{
Joontae Kim
\thanks{Department of Mathematics and Research Institute of Mathematics,  Seoul National University, 1, Gwanak-ro, Gwanak-gu, 08826, Seoul, Republic of Korea.\newline \emph{E-mail address} : \texttt{joontae@snu.ac.kr}}
\and
Myeonggi Kwon
\thanks{Mathematisches Institut, Universit\"at Heidelberg, Im Neuenheimer Feld 205, 69120 Heidelberg, Germany. \newline \emph{E-mail address} : \texttt{mkwon@mathi.uni-heidelberg.de}}
\and
Junyoung Lee
\thanks{\emph{E-mail address} : \texttt{leejunyoung0217@gmail.com}}}
\date{}

%% 컨텐츠 만드는 것. 
\maketitle
%% 컨텐츠에서 표현될 섹션의 단계 설정
\setcounter{tocdepth}{2}
\numberwithin{equation}{section}

%%%%%%%%%%%%%%%%%%%%%%%%%%%%%%%%%%%%%%%%%%%%%%%%%%%%%%%%%%%%%%%%%
%%%%%%%%%%%%%%%%%%%%%%% Abstract
%%%%%%%%%%%%%%%%%%%%%%%%%%%%%%%%%%%%%%%%%%%%%%%%%%%%%%%%%%%%%%%%%

\begin{abstract}
For a Liouville domain $W$ whose boundary admits a periodic Reeb flow, we can consider the connected component $[\tau] \in \pi_0(\symp^c(\widehat W))$ of fibered twists. In this paper, we investigate an entropy-type invariant, called the slow volume growth, of the component $[\tau]$ and give a uniform lower bound of the growth using wrapped Floer homology. We also show that $[\tau]$ has infinite order in $\pi_0(\symp^c(\widehat W))$ if there is an admissible Lagrangian $L$ in $W$ whose wrapped Floer homology is infinite dimensional.

We apply our results to fibered twists coming from the Milnor fibers of $A_k$-type singularities and complements of a symplectic hypersurface in a real symplectic manifold. They admit so-called real Lagrangians, and we can explicitly compute wrapped Floer homology groups using a version of Morse-Bott spectral sequences.\end{abstract}

%\tableofcontents

%%%%%%%%%%%%%%%%%%%%%%%%%%%%%%%%%%%%%%%%%%%%%%%%%%%%%%%%%%%%%%%%%
%%%%%%%%%%%%%%%%%%%%%%% Introduction \label{intro}
%%%%%%%%%%%%%%%%%%%%%%%%%%%%%%%%%%%%%%%%%%%%%%%%%%%%%%%%%%%%%%%%%

\section{Introduction} \label{intro}

One suggestive notion to measure the complexity of a compactly supported diffeomorphism is \emph{volume growth}. Let $(M, g)$ be a Riemannian manifold. Given a compactly supported diffeomorphism $\phi \in \Diff^c(M)$, we define the {\it $i$-dimensional slow volume growth $s_i(\phi)$} by
$$s_i(\phi)=\sup_{\sigma \in \Sigma_i} \liminf_{m \rightarrow \infty} {\log \mu_g(\phi^m (\sigma)) \over \log m} \in [0, \infty]$$
where $\Sigma_i$ is the set of all smooth embeddings of the $i$-dimensional cube $[0,1]^i$ into $M$, and $\mu_g$ is the volume measure induced by the metric $g$. Note that $s_i(\phi)$ is independent of the choice of the metric. The growth $s_i(\phi)$ measures how quickly the volume of $i$-dimensional cubes increases under the iteration of the map. In particular, the \emph{slow} volume growth detects this in a polynomial rate. 

In symplectic geometry, the slow volume growth has provided interesting results on symplectomorphisms. Let $(M, \ow)$ be a symplectic manifold. We denote by $\symp^{(c)}(M, \omega)$ the group of (compactly supported) symplectomorphisms and by $\symp^{(c)}_0(M, \omega)$ the connected component of the identity. Polterovich gave in \cite{Pol1} uniform lower bounds of the 1-dimensional growth $s_1(\phi)$ for every $\phi \in \symp_0(M, \omega) \setminus \{\id\}$ where $(M, \omega)$ is a closed symplectic manifold with vanishing second homotopy group. Frauenfelder and Schlenk proved in \cite{FS1} that $s_1(\phi) \ge 1$, for every symplectomorphism $\phi \in \symp_0^c(T^* B) \setminus \{\id \}$ where $B$ is a closed manifold. 

% 이 문장이 오해하기 쉽게 적혀있음.
The first result that studies the volume growth of elements which are not in the identity component was given in Frauenfelder-Schlenk \cite{FS05}. They considered the generalized Dehn-Seidel twist $\theta$ on the cotangent bundle $T^* B$ of an $n$-dimensional compact rank one symmetric space $B$, shortly CROSS. They proved that $s_n(\phi) \ge 1$ for any $\phi \in \symp^c(T^* B)$ such that $[\phi]=[\theta^k] \in \pi_0(\symp^c_0(T^* B))$ for some $k \in \mathbb{Z} \setminus \{0\}$. An essential idea is that the growth of the intersection number $\#(T_{x_0}^*B \cap \theta^k(T_{x_0}^*B))$ under the iteration of $\theta$ can give a lower bound of the $n$-dimensional slow volume growth. They used Lagrangian Floer homology on cotangent bundles, and a point is that the lower bound is given \emph{uniformly} to the class $[\theta^k]$, due to an invariance property of the Floer homology combined with a geometric argument; such a lower bound for $\theta^k$ is more or less direct from the definition. See also \cite{FLS} and \cite{FS06} for related topics on the slow volume growth. Recently, Alves and Meiwes studied a volume growth of Reeb flows on contact manifolds in \cite{AM}.

In this paper, we extend the idea in \cite{FS05} to more general class of Liouville domains $W$ with a \emph{periodic Reeb flow} on the boundary, in terms of wrapped Floer homology. Denoting by $\widehat{W}$ the completion of $W$ (Section \ref{WFH}), we can define a symplectomorphism called a \emph{fibered twist} $\tau:\widehat{W}\to\widehat{W}$ using the periodic Reeb flow, and cotangent bundles over CROSSes form a special case of these Liouville domains. 

Given an \emph{admissible} Lagrangian $L$ in $W$ see Definition \ref{def: admLag}, we can consider the wrapped Floer homology $\WFH_*(L; W)$ over $\Z_2$-coefficients. We say that $\WFH_*(L; W)$ has a \emph{linear growth} if its dimension increases linearly along the action, see Definition \ref{def: lineargrowth}. A linear growth of $\WFH_*(L; W)$ provides a uniform lower bound of the slow volume growth $s_n$ in the following sense.

\begin{theorem} \label{ThmA} \label{main thm}
Let $(W^{2n}, \lambda)$ be a Liouville domain with a periodic Reeb flow on the boundary and assume that $H^1_c(W; \mathbb{R})=0$. Suppose there exists an admissible Lagrangian ball $L$ such that $\WFH_*(L; W)$ has a linear growth. Then the $n$-dimensional slow volume growth satisfies that $s_n(\phi) \ge 1$ for any compactly supported symplectomorphism such that $[\phi]=[\tau^k] \in \pi_0(\symp^c(\widehat{W}))$ for some $k \ne 0$.
\end{theorem}

In other words, under the existence of a Lagrangian in the theorem, every compactly supported symplectomorphism which is compactly supported symplectically isotopic to (some power of) fibered twists has $n$-dimensional slow volume growth larger than or equal to 1. 

In fact, if $\WFH_*(L; W)$ is infinite dimensional, which is obviously true when $\WFH_*(L; W)$ has a linear growth, the components of $\tau^k$'s in $\symp^c(\widehat{W})$ are all distinct to each other:

\begin{theorem}\label{ThmB}
Let $(W, \lda)$ be as in Theorem \ref{ThmA} and let $\tau:\widehat{W}\to \widehat{W}$ be a fibered twist. Assume that $H_c^1(W;\R)=0$. If there exists an admissible Lagrangian $L$ in $W$ such that $\WFH_*(L;W)$ is infinite dimensional, then $[\tau]$ represents a class of infinite order in $\pi_0(\symp^c(\widehat{W}))$. 
\end{theorem}
We mention that a similar result using symplectic homology can be found in \cite[Corollary 1.2]{Igo} and Theorem \ref{ThmB} can be regarded as the relative version of that result.

We mainly study two classes of examples that Theorem \ref{ThmA} and Theorem \ref{ThmB} apply; \emph{Milnor fibers of $A_k$-type singularities} in Section \ref{sec: A_k} and \emph{complements of symplectic hypersurfaces} in Section \ref{sec: chyper}. 

A Milnor fiber $V$ of an $A_k$-type singularity is a regular level set of a polynomial of the form $f(z) = z_0^{k+1} + z_1^2 + \cdots + z_n^2$ in $\C^{n+1}$. Milnor fibers can be seen as a completion of a Liouville domain $W : = V \cap B^{2n+2}$. The contact type boundary $\partial W$ admits a canonical periodic Reeb flow, which gives rise to a fibered twist $\tau$ on $V$. We give a uniform lower bound of the $n$-dimensional slow volume growth of the component of the fibered twists $[\tau] \in \pi_0(\symp^c(V))$.
 
For $k=1$, the fibered twist on $V$ is the same as the twist on the cotangent bundle $T^*S^n$ studied in \cite{FS05}. For general $k$, it is well-known that the Milnor fiber is Stein deformation equivalent to the $k$-fold linear plumbing of $T^*S^n$'s, see \cite[Section 3.3]{Sei3}. Our result can be seen as a partial generalization of the results in \cite{FS05} to the linear plumbings. Let us also point out that for $n$ even, the fibered twists on $A_k$-Milnor fibers are \emph{smoothly} isotopic to the identity. In this sense, the (positive) uniform lower bound of the slow volume growth captures a genuine symplectic phenomenon for those cases.

A nice feature of the Milnor fibers for our purpose is that they are \emph{real} Liouville domains; they admit a canonical \emph{anti-symplectic involution} (\ref{eq: conjmap}), essentially given by the complex conjugate. Its fixed point set is then a Lagrangian, called a \emph{real Lagrangian}. It turns out that (a connected component of)  the fixed point set is an admissible Lagrangian and is diffeomorphic to the ball. This allows us to apply Theorem \ref{ThmA} and \ref{ThmB}. Moreover we can compute the wrapped Floer homology groups $\WFH_*(L;W)$ explicitly, see Corollary \ref{cor: WFHcomp}.
%We consider the corresponding fibered twists and show a uniform lower bound of the volume growth of their components in $\pi_0(\symp^c(\widehat{W}))$. We note that the result on $A_k$-Milnor fibers generalizes the $T^*S^n$-case in \cite{FS05} to its $A_k$-type \emph{plumbings}.

Another interesting examples we consider are complements of a symplectic hypersurface in a closed symplectic manifold. They are Liouville domains where the boundary forms a \emph{prequantization bundle} in the sense that all Reeb orbits are periodic with the same period. Let $M$ be a closed symplectic manifold and $Q \subset M$ a symplectic hypersurface. We additionally give an anti-symplectic involution on $M$, and $Q\subset M$ is chosen to be invariant under the involution. The complement $W=M\setminus \nu_M(Q)$ of a neighborhood of $Q$ in $M$ is then a real Liouville domain (Proposition \ref{comp_prop}). For computational purposes, we further assume technical setup given in Section \ref{comp_setup}, see \cite{Di} and \cite[Section 4.2]{Ka} for similar setups. 

We provide two explicit examples of the setup; a complement of a smooth complex hypersurface of degree $k$ in the complex projective space $\C P^n$ in Section \ref{example_cp} and a complement of a hyperplane section in a complex hypersurface in Section \ref{sec: comp_ex2}. Similarly to the Milnor fiber case, it turns out that (a connected component of) the real Lagrangian in the complements is diffeomorphic to the ball, and its wrapped Floer homology group can be computed, see Corollary \ref{comp_exam1cor} and Corollary \ref{comp_exam2cor}. Since it has a linear growth, we can apply Theorem \ref{ThmA} and \ref{ThmB}. We consequently obtain a uniform lower bound of the $n$-dimensional slow volume growth of the components of (a power of) a fibered twist on the complements, and we also find that they have infinite order in $\pi_0(\symp^c(\widehat{W}))$.

%여기서부터, 8.6.2에 있는 구체적인 예와 더불어 라그랑지안이 어떻게 주어지고, 따라서 ThmA와 ThmB를 적용할 수 있는 상황으로 되고, 따라서 역시 slow volume growth와 infinite order를 증명할 수 있는데, 더욱이 구체적인 WFH계산도 가능하다. Proposition \ref{comp_prop}에서 처럼. 이렇게 적으면어떨까? 그럼 다로 다음에 이어지는 계산 툴에 대한 이야기랑 잘 맞을 듯.

Our main tool for computations of wrapped Floer homology is a version of a Morse-Bott spectral sequence, stated in Section \ref{sec: MBss}. Since the spectral sequence is constructed using an action filtration, it tells us not only about the group structure but about the limit (\ref{growthrateintro}). Looking at periodicity of the first page together with the resulting homology, we obtain a linear growth of wrapped Floer homology. 

In practice, it is necessary to compute Maslov indices of Morse-Bott component of Reeb chords. We do this in a general setting of real Liouville domains with a periodic Reeb flow. In Proposition \ref{lem: indrel}, we give a relation between Maslov indices of periodic Reeb orbits and its \emph{half} Reeb chords when the period of Reeb orbits are given by the minimal common period.

%In this paper, we consider the following classes of Liouville domains with periodic Reeb flow at the boundary.
%\begin{enumerate}[label=\arabic*)]
%\item The cotangent bundle $S^* N$ over a CROSS $(N, g)$,
%\item $A_k$-type Milnor fibers,
%\item Complement of a symplectic hypersurface in an integral symplectic manifolds.
%\end{enumerate}

%Even though case 1) has been proven in \cite{FS05}, we will re-prove this case by showing this case fits into the conditions of Theorem \ref{ThmA}. As corollaries of Theorem \ref{ThmA}, we will give various examples in case 2) and 3).

\subsection*{Acknowledgments} We deeply thank to Otto van Koert for helpful suggestions and discussions. The first author thanks to Urs Frauenfelder and University of Augsburg for warm hospitality. He was partially supported by the NRF Grant NRF-2016R1C1B2007662 funded by the Korean government. The second author was partially supported by SFB/TRR 191 - Symplectic Structures in Geometry, Algebra and Dynamics. The third author was partially supported by the European Research Council (ERC) under the European Union's Horizon 2020 research and innovation programme, starting grant No. 637386

%%%%%%%%%%%%%%%%%%%%%%%%%%%%%%%%%%%%%%%%%%%%%%%%%%%%%%%%%%%%%%%%%
%%%%%%%%%%%%%%%%%%%%%%% Wrapped Floer homology \label{WFH}
%%%%%%%%%%%%%%%%%%%%%%%%%%%%%%%%%%%%%%%%%%%%%%%%%%%%%%%%%%%%%%%%%

\section{Floer theory}
\subsection{Wrapped Floer homology}\label{WFH}

We give a definition of wrapped Floer homology, which is basically an ``open-string analogue'' of symplectic homology. We refer to \cite[Section 3]{AboSei} and \cite[Section 4]{Rit} for more detailed description. Note that we always use $\Z_2$-coefficients in this paper.

Let $(W^{2n}, \lambda)$ be a Liouville domain with a Liouville form $\lambda$. The Liouville vector field $X_\lambda$ of $\lambda$ is given by $d\lambda(X_\lambda,\cdot)=\lambda$ and $X_\lambda$ is pointing outwards along the boundary. The restriction $\alpha:=\lambda|_{\partial W}$ is a contact form on the boundary $\p W$. We write $R_\alpha$ for the Reeb vector field on $\p W$ characterized by the conditions $d\alpha(R_\alpha,\cdot)=0$ and $\alpha(R_\alpha)=1$. We define the {\it completion} $\widehat{W}$ of a Liouville domain $W$ by $\widehat{W}:=W \cup_{\partial W} \left([1, \infty) \times \partial W \right)$ where we attach the cylindrical part $[1, \infty) \times \partial W$ to $W$ along the boundary using the Liouville flow. We also extend the Liouville form $\lambda$ by
$$
\widehat{\lambda}=\begin{cases}  \lambda & \mbox{ in } W, \\ r \alpha & \mbox{ in } [1, \infty) \times \partial W,\end{cases}
$$
where $r$ stands for the coordinate on $[1, \infty)$.
Note that the Liouville vector field $X_{\widehat{\lambda}}$ restricts to $r \p_r$ on $[1,\infty)\times \p W$.
\begin{Def}\label{def: admLag}
A Lagrangian $L$ in $W$ is called {\it admissible} if
\begin{itemize}
	\item $L$ is transverse to the boundary $\partial W$ and the intersection $\p L=L \cap \partial W$ is a Legendrian,
	\item $L$ is exact, i.e., $\lambda|_L$ is an exact 1-form,
	\item the Liouville vector field of $\lambda$ is tangent to $L$ near its  boundary. 
\end{itemize}
\end{Def}
Since the Liouville vector field is tangent to $L$ near its boundary, we can complete an admissible Lagrangian $L$ to a noncompact Lagrangian $\widehat{L}=L \cup_{\partial L} \left([1, \infty) \times \partial L \right)$ in $\widehat{W}$. 

Let $H:\widehat{W}\to \R$ be a Hamiltonian and $X_H$ the associated Hamiltonian vector field defined by $d\widehat{\lambda}(X_H,\cdot)=dH$. We denote by $Fl^{X_H}_t:\widehat{W}\to \widehat{W}$ the Hamiltonian flow of $X_H$. A trajectory $x: [0,1] \rightarrow \widehat W$ of the Hamiltonian vector field $X_H$ such that $x(0), x(1) \in \widehat L$ is called a \textit{Hamiltonian chord} of $H$. The Hamiltonian chords corresponds to the intersections $L\cap Fl_1^{X_H}(L)$. We say that a Hamiltonian chord $x$ is \textit{contractible} if the class $[x] \in \pi_1(\widehat W, \widehat L)$ is trivial. We denote by $\mathcal{P}(H)$ the set of all contractible Hamiltonian chords of $H$.
\begin{Def}
	A Hamiltonian chord $x$ is called {\it nondegenerate} if 
	$$
	TFl_1^{X_H}(T_{x(0)}\widehat{L})\cap T_{x(1)}\widehat{L}=\{0\}.
	$$
	A Hamiltonian $H:\widehat{W}\to \R$ is called {\it nondegenerate} if all Hamiltonian chords of $H$ are nondegenerate, i.e., $\widehat{L}$ and $Fl_1^{X_H}(\widehat{L})$ intersect transversally.
\end{Def}
A {\it Reeb chord} of period $T>0$ is a trajectory $c:[0,T]\to \p W$ of the Reeb vector field $R_\alpha$ such that $c(0),c(T)\in \p L$.  We denote by $\Spec(\p W,\alpha,\p L)$ the set of all periods of Reeb chords.
\begin{Def}
A time-independent Hamiltonian $H: \widehat W \rightarrow \R$ is called \textit{admissible} if  $H(r, x) = ar+b$ on the symplectization $([1,\infty)\times \p W,r\alpha)$ with $a \not \in \Spec(\partial W, \alpha, \partial L)$. Here $a$ is called the \textit{slope} of $H$.
\end{Def}
For any admissible Hamiltonian $H$, it is known that a generic perturbation of $H$ (while fixing the slope) is nondegenerate, see \cite[Lemma 8.1]{AboSei}. Since the slope of $H$ does not lie in $\Spec(\p W,\alpha,\p L)$, all Hamiltonian chords are contained in a compact domain $W$ in $\widehat{W}$. 

Under the assumption that the Maslov class $\mu_{\widehat{L}}:\pi_2(\widehat{W},\widehat{L})\to \Z$ of a Lagrangian $\widehat{L}$ vanishes, we can assign the Maslov index of a contractible Hamiltonian chord $x:[0,1]\to \widehat{W}$ as follows. Let $D^+=\{z\in \mathbb{C} \ |\  |z|\le 1,\ \im z\ge 0\}$ be a half-disk and let $D^+_\R:=D^+\cap\R$ be the real line in $D^+$. A smooth map $v:(D^+,D^+_\R)\to (\widehat{W},\widehat{L})$ is called a {\it capping half-disk} of a Hamiltonian chord $x:[0,1]\to \widehat{W}$ if $v(e^{\pi it})=x(t)$ for $t\in[0,1]$ and $v(D^+_\R)\subset \widehat{L}$. Let $\Lambda_\text{hor}^n=\{z\in \C^n \ |\ \im z=0\}$ be the horizontal Lagrangian subspace in $(\C^n,\ow_{\text{std}}:=\frac{i}{2}\sum_jdz_j\wedge d\bar{z}_j)$. Note that a capping half-disk $v:D^+\to \widehat{W}$ of $x$ yields a  symplectic trivialization of $v^*T\widehat{W}$
$$
\phi_v:(v^*T\widehat{W},d\widehat{\lambda})\longrightarrow D^+\times (\mathbb{C}^n,\ow_\text{std})
$$
such that $\phi_v(T_{v(z)}\widehat{L})=\Lambda_\text{hor}^n$ for $z\in D^+_\R$. We call such a trivialization $\phi_v$ an {\it adapted symplectic trivialization} of $v^*T\widehat{W}$. Note that $\phi_v$ is unique up to homotopy of adapted symplectic trivializations and a multiplication by some fixed matrix in $\Sp(2n)^{\Lambda_\text{hor}^n}:=\{\Psi\in \Sp(2n)\ |\ \Psi\Lambda_\text{hor}^n=\Lambda_\text{hor}^n \}$ which is not contained in the identity component of $\Sp(2n)^{\Lambda_\text{hor}^n}$. We abbreviate
$$
\Psi_x:[0,1]\to \Sp(2n),\quad \Psi_x(t):=\phi_v(x(t))\circ TFl_t^{X_H}(x(0))\circ \phi_v^{-1}(x(0))
$$
the linearization of the Hamiltonian flow along $x$ with respect to $\phi_v$. Then, we obtain a path of Lagrangian subspaces in $(\C^n,\ow_\text{std})$,
$$
t\in [0,1]\longmapsto \Psi_x\Lambda_\text{hor}^n\in \mathcal{L}(n),
$$
where $\mathcal{L}(n)$ is the space of Lagrangian subspaces in $\C^n$.
In the following, $\mu_{RS}$ denotes the Robbin-Salamon index defined in \cite{RS}.
\begin{Def}
The {\it Maslov index} of a contractible Hamiltonian chord $x:[0,1]\to \widehat{W}$ is defined by
	$$
	\mu(x):=\mu_{RS}(\Psi_x\Lambda_\text{hor}^n,\Lambda_\text{hor}^n).
	$$
\end{Def}
\begin{Rem}\label{rem: hamindex}
By property of stratum homotopy of Robbin-Salamon index \cite[Theorem 2.4]{RS}, $\mu(x)$ is independent to the choice of adapted symplectic trivializations.
\end{Rem}
We want to show that $\mu(x)$ does not depend on the choice of capping half-disk. The following lemma seems to be known, but we could not find any references. Hence, we give a detailed proof. 

\begin{Lem}\label{lem: cappingind}
Let $u,v:D^+\to \widehat{W}$ be capping half-disks of $x$. Then  
we have
$$
\mu(x;u)-\mu(x;v)=\mu_{\widehat{L}}([u\#\overline{v}]),
$$
where $u\#\overline{v}:(D^2,\p D^2)\to (\widehat{W},\widehat{L})$ is a disk obtained by gluing $u$ and $\overline{v}$ along $x(t)$. Here $\overline{v}$ is a capping half-disk with a reversed orientation of $v$.
\end{Lem}
\begin{proof}
Let $\phi_u$ and $\phi_v$ be adapted symplectic trivializations of $u^*T\widehat{W}$ and $v^*T\widehat{W}$, respectively. For simplicity, write $\Lambda=\Lambda_\text{hor}^n$.
By \cite[Corollary 3.14]{FK}, we obtain
$$
\mu(x;u)-\mu(x;v)=\mu_{RS}(\phi_u(t)\circ \phi_v^{-1}(t)\Lambda,\Lambda),
$$
where $\phi_u(t)\circ \phi_v^{-1}(t):=\phi_u(x(t))\circ \phi_v^{-1}(x(t)):[0,1]\to \Sp(2n)$ is a path of symplectic matrices given by the transition matrices along $x(t)$. It remains to prove that
$$
\mu_{RS}(\phi_u(t)\circ \phi_v^{-1}(t)\Lambda,\Lambda)=\mu_{\widehat{L}}([u\# \overline{v}]).
$$
To see that, choose a homotopy of symplectic trivializations of $v^*T\widehat{W}$,
$$
\phi^s_v:v^*T\widehat{W}\to D^+\times \C^n,\quad s\in [0,1]
$$
such that $\phi^0_v=\phi_v$, $\phi^1_v(x(t))=\phi_u(x(t))$ for $t\in [0,1]$, and $\phi^s_v(T_{x(0)}L)=\Lambda$ for all $s\in [0,1]$. Let $\psi_v:=\phi^1_v$ be the symplectic trivialization of $v^*T\widehat{W}$. Write $\sigma:=u\#\overline{v}:(D^2,\p D^2)\to (\widehat{W},\widehat{L})$. We can glue trivializations $\phi_u$ and $\psi_v$ along $x(t)$ so that we obtain a symplectic trivialization
$$
\phi_\sigma:\sigma^*T\widehat{W}\to D^2\times \C^n.
$$
%Here we choose a specific parametrization of $\p \sigma(t)$ such that $\p \sigma(0)=x(0)$. (Hence, we know $\phi_\sigma(T_{\p \sigma(0)}L)=\Lambda$.)
We consider paths of Lagrangian subspaces
\begin{eqnarray*}
\overline{\Phi}^s_v:[0,1]\to \mathcal{L}(n),&& \quad t\mapsto \phi_v^s(T_{v(1-2t)}L),\quad s\in [0,1] \\
{\Phi}_u:[0,1]\to \mathcal{L}(n),&& \quad t\mapsto\phi_u(T_{v(2t-1)}L)	.
\end{eqnarray*}
In the below, the notation $\#$ means a concatenation of paths of Lagrangian subspaces. Note that
$$
\Psi_s:=[\phi_u(t)\circ ({\phi_v^s})^{-1}(t)\Lambda ]\# \overline{\Phi}_v^s\# \Phi_u,\quad s\in [0,1]
$$
is a homotopy of loops of Lagrangian subspaces.
Since $\overline{\Phi}^{0}_v\#\Phi_u$ is constant, we observe that
$$
\mu_{RS}(\phi_u(t)\circ \phi_v^{-1}(t)\Lambda,\Lambda)=\mu_{RS}(\Psi_0,\Lambda).
$$
By definition of Maslov class, we get
$
\mu_{\widehat{L}}([\sigma])=\mu_{RS}(\phi_\sigma(T_{\p \sigma(t)}L),\Lambda)=\mu_{RS}(\Psi_1,\Lambda).
$
\end{proof}
As a result, the Maslov index of a contractible Hamiltonian chord is well-defined provided that the Maslov class $\mu_{\widehat{L}}$ vanishes.
\begin{Rem}\label{rem: topcondmas}
In this paper, we will use the following topological assumption:
$c_1(W)$ vanishes on $\pi_2(W)$ and $H_1(L)=0$. This implies that the Maslov class $\mu_{\widehat{L}}$ vanishes. See \cite[Lemma 2.1]{Oh1}.
\end{Rem}

Recall that $\widehat{L}$ is exact, so $\widehat{\lambda}|_{\widehat{L}}=df$ for some $f\in C^\infty(\widehat{L})$. The {\it action functional} of $H$ is defined by
\begin{eqnarray*}
\mathcal{A}_H(x) &=& -\int_{D^+}v^*d\widehat{\lambda}-\int_0^1 H(x(t))dt \\
&=& f(x(1))-f(x(0))-\int_0^1x^*\widehat{\lambda}-\int_0^1H(x(t))dt,
\end{eqnarray*}
where $v:D^+\to \widehat{W}$ is a capping half-disk of a contractible path $x:[0,1]\to \widehat{W}$ such that $x(0),x(1)\in \widehat{L}$. Then, the critical points of $\mathcal{A}_H$ are exactly given by  contractible Hamiltonian chords.

Now, we assume that $H$ is a nondegenerate admissible Hamiltonian and $J=\{J_t\}_{t\in[0,1]}$ is a smooth family of compatible almost complex structures on $\widehat{W}$ which are of {\it contact type} for large $r\gg1$, see \cite[Section 3.2]{AboSei} for definition. With this choice of almost complex structures, we can apply the maximum principle, see \cite[Lemma 4.4]{Rit}.
\begin{Rem}
	One can also start with an admissible Hamiltonian and perturb it to being nondegenerate. It is known that the resulting homology will not depend on the generic perturbation.
\end{Rem}
For $x,y\in \mathcal{P}(H)$ with $x\ne y$, we define the {\it moduli space of Floer strips} from $x$ to $y$ as
\begin{eqnarray*}
\mathcal{M}(x,y)=\{u:\R\times [0,1]\to \widehat{W} &|&  \p_s u+J_t(\p_t u -X_H(u))=0,  \\
&& \lim_{s\to -\infty}u(s,t)=x(t),\\
&& \lim_{s\to \infty}u(s,t)=y(t),\\
&& u(s,0),u(s,1)\in \widehat{L}\} /\R,
\end{eqnarray*}
where the group $\R$ acts on $\mathcal{M}(x,y)$ by the translations in the $s$-variable. Note that the action decreases along Floer strips. It is known that $\mathcal{M}(x,y)$ is a smooth manifold of dimension $|x|-|y|-1$ for generic $J$, see \cite{FHS} and \cite{RS2}.

For $x\in \mathcal{P}(H)$, we define the {\it index} of $x$ by $|x|=\mu(x)-\frac{n}{2}$. If $x$ is nondegenerate, then we have $|x|\in \Z$. Let $c\in \R\cup \{\infty\}$. The {\it wrapped Floer chain group} is defined by
$$
\WFC_*^{<c}(L;H)=\bigoplus_{\substack{ x\in \mathcal{P}(H)\text{ with }|x|=* \\ \mathcal{A}_H(x)<c}}\Z_2\langle x\rangle
$$
with the {\it boundary map}
$$
\p: \WFC_*^{<c}(L;H)\to \WFC_{*-1}^{<c}(L;H), \quad \p(x)=\sum_{|x|-|y|=1}\#_{\Z_2}\mathcal{M}(x,y)\cdot y.
$$
Here, $\#_{\Z_2}\mathcal{M}(x,y)$ means the number of elements in $\mathcal{M}(x,y)$ mod 2.
For generic choice of $J$, the boundary map is well-defined and $\partial^2=0$. We define the {\it wrapped Floer homology of a Hamiltonian $H$} as the homology of the chain complex
$$
\WFH_*^{<c}(L;H)=H_*(\WFC_*^{<c}(L;H),\partial).
$$
For generic pairs $(H_\pm, J_\pm)$ with $H_-\le H_+$, we define a chain map
$$
\Phi^{(H_-,H_+)}:\WFC_*^{<c}(L;H_-)\to \WFC_*^{<c}(L;H_+)
$$
by counting certain $s$-dependent Floer strips, see \cite[Section 3.2]{Rit} for the construction. This induces a {\it continuation map} $\Phi_*^{(H_-,H_+)}:\WFH_*^{<c}(L;H_-)\to \WFH_*^{<c}(L;H_+)$. Note that continuation maps satisfy the functorial property as follows: for any admissible Hamiltonians $H_0\le H_1\le H_2$, we have
$$
\Phi^{(H_1,H_2)}_*\circ \Phi^{(H_0,H_1)}_*=\Phi^{(H_0,H_2)}_* \quad \text{and}\quad
\Phi_*^{(H_0,H_0)}=\id.
$$
A standard argument in Floer theory shows that $\WFH_*^{<c}(L;H)$ does not depend on the choice of $J$ and it only depends on the slope of $H$, see \cite[Lemma 3.1]{Rit}. Finally, we define the {\it wrapped Floer homology groups of a Lagrangian $L$} to be
$$
\WFH_*^{<c}(L;W):=\lim_{\underset{H}{\longrightarrow}} \WFH_*^{<c}(L;H),
$$
where the direct limit is taken over admissible Hamiltonians whose slope diverges to $\infty$. If $c=\infty$, then we set $\WFH_*(L;W):=\WFH_*^{<\infty}(L;W)$.

\subsection{Morse-Bott spectral sequence in wrapped Floer homology} \label{sec: MBss}
Our computational tool for the wrapped Floer homology group $\WFH_*(L;W)$ is a version of Morse-Bott spectral sequences, and it can be described in the setting of \emph{Morse-Bott type Reeb chords}. 
\subsubsection{Maslov index of Reeb chords}
Let $(\Sigma^{2n-1},\alpha)$ be a contact manifold and $\mathcal{L}$ a Legendrian. Under the assumption that the Maslov class $\mu_\mathcal{L}:\pi_2(\Sigma,\mathcal{L})\to \Z$ of a Legendrian $\mathcal{L}$ vanishes, we define Maslov index of contractible Reeb chords which is parallel to the case of Hamiltonian chords as in Section \ref{WFH}. A Reeb chord $c:[0,T]\to \Sigma$ is called \textit{contractible} if $[c]=0\in \pi_1(\Sigma,\mathcal{L})$. A smooth map $v:(D^+,D^+_\R)\to (\Sigma,\mathcal{L})$ is called a {\it capping half-disk} of a Reeb chord $c:[0,T]\to \Sigma$ if $v(e^{\pi it/T})=c(t)$ for $t\in [0,T]$ and $v(D^+_\R)\subset \mathcal{L}$. Let $\Lambda_\text{hor}^{n-1}$ be the horizontal Lagrangian subspace in $(\C^{n-1},\ow_{\text{std}})$. Choose an {\it adapted symplectic trivialization} of $v^*\xi$
$$
\phi_v:(v^*\xi,d\alpha)\longrightarrow D^+\times (\C^{n-1},\ow_\text{std})
$$
meaning that $\phi_v(T_{v(z)}\mathcal{L})=\Lambda_\text{hor}^{n-1}$ for $z\in D^+_\R$. Abbreviating
$$
\Psi_c:[0,T]\to \Sp(2n-2),\quad \Psi_c(t):=\phi_v(x(t))\circ TFl_t^{R_\alpha}|_\xi(c(0))\circ \phi_v^{-1}(x(0))
$$
the linearization of the Reeb flow along $c$ with respect to $\phi_v$, we obtain a path of Lagrangian subspaces in $(\C^{n-1},\ow_\text{std})$,
$$
t\in [0,T]\longmapsto \Psi_c\Lambda_\text{hor}^{n-1}.
$$

\begin{Def}
The \textit{Maslov index} of a contractible Reeb chord $c:[0,T]\to \Sigma$ is defined by
	$$
	\mu(c):=\mu_{RS}(\Psi_c\Lambda_\text{hor}^{n-1},\Lambda_\text{hor}^{n-1}).
	$$
\end{Def}
\begin{Rem}\label{rem: topcondmasReeb}
Since the Maslov class $\mu_\mathcal{L}$ vanishes, $\mu(c)$ is independent of the choices involved. See Remark \ref{rem: hamindex} and Lemma \ref{lem: cappingind}. It is known that $\mu_\mathcal{L}$ vanishes if $c_1(\xi)$ vanishes on $\pi_2(\Sigma)$ and $H_1(\mathcal{L})=0$.
\end{Rem}

\subsubsection{Morse-Bott type Reeb chords}
Let $A,B$ be submanifolds of a smooth manifold $X$. We say that $A$ and $B$ \textit{intersect cleanly along} $A\cap B$ if $A\cap B$ is a submanifold of $X$ and $T_pA\cap T_pB=T_p(A\cap B)$ holds for all $p\in A\cap B$.
\begin{Def}\label{def: morsebott}
Let $(\Sigma,\alpha)$ be a contact manifold and $\mathcal{L}\subset \Sigma$ a Legendrian. We say that Reeb chords are of \textit{Morse-Bott type} if the following holds.
\begin{enumerate}[label=(\arabic*)]
	\item The spectrum of Reeb chords $\text{Spec}(\Sigma,\alpha,\mathcal{L})$ is discrete.
	\item \label{msReebchord} For $T\in \text{Spec}(\Sigma,\alpha,\mathcal{L})$, the set
\begin{equation}\label{eq: mbcp}
\mathcal{L}_T: = \{z \in \mathcal{L} \;|\; Fl_T^{R_{\alpha}}(z) \in \mathcal{L}\}	
\end{equation}
is a (possibly disconnected) closed submanifold of $\mathcal{L}$ such that $Fl_{-T}^{R_\alpha}(\mathcal{L})$ and $\mathcal{L}$ intersect cleanly along $\mathcal{L}_T$, i.e.,
$$
T_z\mathcal{L}_T=T_z(Fl_{-T}^{R_\alpha}(\mathcal{L}))\cap T_z\mathcal{L}=TFl_{-T}^{R_\alpha}(T_w\mathcal{L})\cap T_z\mathcal{L}\quad \text{for $z\in \mathcal{L}_T$ and $w=Fl_T^{R_\alpha}(z)$}.
$$\end{enumerate}
Each connected component of $\mathcal{L}_T$ is called a \textit{Morse-Bott component}. 
\end{Def}
\begin{Rem}
Any point in a Morse-Bott component $\mathcal{L}_T$ gives rise to a corresponding Reeb chord of period $T$ starting at the point. Hence, we can regard $\mathcal{L}_T$ as a family of Reeb chords of period $T$ parametrized by starting points. By a Reeb chord in $\mathcal{L}_T$ we mean that a Reeb chord of period $T$ starting at a point in $\mathcal{L}_T$. We also write $\mathcal{L}_T$ for a Morse-Bott component unless it causes confusion.
\end{Rem}
A Morse-Bott component $\mathcal{L}_T$ is called {\it contractible} if a Reeb chord in $\mathcal{L}_T$ is contractible. Note that every Reeb chord in $\mathcal{L}_T$ is contractible if one of them is so. The {\it Maslov index} of a contractible Morse-Bott component $\mathcal{L}_T$ is defined by
$$
\mu(\mathcal{L}_T):=\mu(c),
$$
where $c$ is a Reeb chord in $\mathcal{L}_T$. This does not depend on the choice of the Reeb chord by property of stratum homotopy \cite[Theorem 2.4]{RS} together with the condition \ref{msReebchord} in Definition \ref{def: morsebott}. 

\subsubsection{Morse-Bott spectral sequence}
Let $(W,\lambda)$ be a Liouville domain and let $L\subset W$ be an admissible Lagrangian. Denote by $\mathcal{L}:=\partial L$ a Legendrian. We assume that Reeb chords are of Morse-Bott type. For well-definedness of indices, we further assume that
\begin{enumerate}[label=(\arabic*)]
	\item \label{ms1} $c_1(W)$ vanishes on $\pi_2(W)$.
	\item \label{ms2} $H_1(L)=H_1(\mathcal{L})=0$.
	\item \label{ms3} The map $i_*:\pi_1(\Sigma,\mathcal{L})\to \pi_1(W,L)$ induced by the inclusion is injective.
\end{enumerate}
\begin{Rem}
The condition \ref{ms1}, \ref{ms2} implies that we have a well-defined $\Z$-grading on $\WFH_*(L;W)$ and Maslov index of Morse-Bott component of Reeb chords are well-defined, see Remark \ref{rem: topcondmas} and \ref{rem: topcondmasReeb}. The condition \ref{ms3} will be used to express the Maslov index of Hamiltonian chords in terms of the Maslov index of Reeb chords, see \cite[Lemma 3.4]{BO} an analogue statement in a symplectic homology.
\end{Rem}
Denote by $\text{Spec}_0(\Sigma,\alpha,\mathcal{L})$ the set of all periods of contractible Reeb chords, and we arrange the spectrum $\text{Spec}_0(\Sigma,\alpha,\mathcal{L})=\{T_1,T_2,\cdots\}$ in such a way that $0<T_1 < T_2 < \cdots$. The following spectral sequence can be constructed using action filtrations adapted to the Morse-Bott setup. 

\begin{Thm}\label{thm: ss}
There is a spectral sequence $(E^r_{pq},d^r)$ converging to $\WFH_*(L;W)$ whose first page is given as follows:
\begin{equation}
\label{eq: MBss}
E^1_{pq} =  \begin{cases} \displaystyle  H_{p+q - \shift(\mathcal{L}_{T_p})+\frac{1}{2}\dim L}(\mathcal{L}_{T_p} ; \Z_2) & p>0, \\ \displaystyle H_{q + \dim L} (L, \mathcal{L}; \Z_2) & p = 0, \\
0 & p < 0.  \end{cases}
\end{equation}
where $\shift(\mathcal{L}_{T_p}) = \mu(\mathcal{L}_{T_p}) - \frac{1}{2} (\dim \mathcal{L}_{T_p} -1)$.
\end{Thm}

\begin{Rem}\label{rem: mbss} \
\begin{itemize}
	\item The differentials are given by $d^r:E^r_{pq} \to E^r_{p-r,q+r-1}.$
	\item Our grading convention for wrapped Floer homology is such that $\WFH_*^{<\epsilon}(L; W) \cong H_{*+n}(L, \mathcal{L})$ for $\epsilon > 0$ sufficiently small.
	\item A detailed description of constructions of Morse-Bott spectral sequences in symplectic homology can be found in \cite[Theorem B.11]{KvK}. See also \cite{Sei2} for a spectral sequence in Lagrangian Floer homology.
\end{itemize}
\end{Rem}

\subsection{Lagrangian Floer homology}
We briefly recall the Lagrangian Floer homology of a pair of Lagrangians in a Liouville domain, based on \cite[Section 5]{MiSei}.  We also refer to \cite[Section 2.4]{FS05}.

Throughout this section, let $(W,\lambda)$ be a Liouville domain and let $L_0,L_1$ be admissible Lagrangians  in $W$ satisfying either $\p L_0=\p L_1$ or $\p L_0\cap \p L_1\ne \emptyset$. Note that $L_0$ and $L_1$ may not intersect transversally. We will consider each case separately.
\subsubsection*{Case 1. $\p L_0 \cap \p L_1=\emptyset$}
We first deform $L_0$ and $L_1$ so that they intersect transversally. More precisely, choose admissible Lagrangians $L_0',L_1'$ such that $L_j$ is exact isotopic to $L_j'$ rel $\p W$, for $j=0,1$, with the property that $L_0'$ and $L_1'$ intersect transversally.  Then, the original construction \cite{Floer} of the Lagrangian Floer homology of a pair $(L_0',L_1')$ works with contact type almost complex structures that guarantees the maximum principle \cite[Lemma 5.5]{MiSei}. We denote by $\LFH_*(L_0',L_1')$ the resulting Lagrangian Floer homology. We then define the {\it Lagrangian Floer homology of a pair} $(L_0,L_1)$ by $\LFH_*(L_0,L_1):=\LFH_*(L_0',L_1')$. It turns out that $\LFH_*(L_0,L_1)$ does not depend on the choices involved. More generally, the following invariance property holds \cite[Proposition 5.10]{MiSei}. See also \cite[Theorem 5.1]{Oh1} for a detailed proof.

\begin{Lem}\label{lem: firstinv}
Let $L_1'$ be an admissible Lagrangian which is exact isotopic to $L_1$ rel $\p W$. Then, we have an isomorphism $\LFH_*(L_0,L_1)\cong \LFH_*(L_0,L_1')$.
\end{Lem}
\subsubsection*{Case 2. $\p L_0 =\p L_1$} 
To reduce this to the case 1, we use a special kind of deformation which is called a \emph{positive} isotopy,  introduced in \cite[Section 5a]{MiSei}.

Let $H:W\to \R$ be a Hamiltonian such that $H(r,x)=ar+b$ on a collar neighborhood $([1-\epsilon,1]\times \p W)\cong \nu_W(\p W)$. The exact isotopy $L_1^t:=Fl_t^{X_H}(L_1)$ induced by the flow of $X_H$ is called a {\it positive isotopy}. Since the Hamiltonian vector field $X_H$ on $[1-\epsilon,1]\times \p W$ is given by $X_H=-aR_\alpha$, it is easy to see that a Lagrangian $L_1^+:=Fl_1^{X_H}(L_1)$ is admissible. We choose a sufficiently small (i.e., the slope of $H$ is small enough) positive exact isotopy $L_1^t$ of $L_1$ such that $\p L_0\cap \p L_1^t=\emptyset$ for all $t\in[0,1]$. In particular, we have $\p L_0\cap\p L_1^+=\emptyset$, and we define $\LFH_*(L_0,L_1):=\LFH_*(L_0,L_1^+)$ following the case 1. 

To show the independence of the choice of the positive deformation $L_1^+$, we use the following generalization of Lemma \ref{lem: firstinv}. See \cite[Lemma 5.11]{MiSei}.
\begin{Lem}\label{lem: secondinv}
	Let $L_0$ be an admissible Lagrangian. Suppose that $(L_1^t)_{t\in[0,1]}$ is an exact isotopy of admissible Lagrangians such that $\p L_0\cap \p L_1^t=\emptyset$ for all $t$. Then, $\LFH_*(L_0,L_1^t)$ is independent of $t$ up to isomorphism.
\end{Lem}
This lemma immediately shows that $\LFH_*(L_0,L_1)$ is well-defined.
\subsection{Properties of the wrapped/Lagrangian Floer homology}
We list well-known properties of the wrapped Floer homology and the Lagrangian Floer homology that will be used to prove Theorem \ref{ThmA} and \ref{ThmB}. In what follows, let $L$ be an admissible Lagrangian in a Liouville domain $(W,\lambda)$.
\begin{enumerate}[label=(P\arabic*)]
		\item \label{p1} For $c\notin \Spec(\p W,\alpha,\p L)$, there exists a nondegenerate admissible Hamiltonian $H:\widehat{W}\to \R$ of slope $c$ such that
$$\WFH_*^{<c}(L;W)=\WFH_*(L;H).$$
		\item \label{p2} Let $H$ be an admissible Hamiltonian (not necessarily nondegenerate). Then, we have an isomorphism
$$
\WFH_*(L;H)\cong \LFH_*(L,Fl^{X_H}_1(L)).
$$
For instance, we refer to \cite[Remark 4.4]{Rit}.
\item \label{p3} If $H$ is nondegenerate, then $L$ and $Fl_1^{X_H}(L)$ are admissible Lagrangians that intersect transversally. By the property \ref{p2}, the dimension of $\WFH_*(L;H)$ gives a lower bound of intersection points of $L\cap Fl_1^{X_H}(L)$.

\item \label{p4} Let $L_0,L_1$ be admissible Lagrangians. For any $\varphi\in \text{Ham}^c(W)$, it follows from Lemma \ref{lem: firstinv} that we obtain the invariance
$$
\LFH_*(L_0,L_1)\cong \LFH_*(L_0,\varphi(L_1))
$$
up to degree shift. Here we define $\ham^c(W)$ the group of Hamiltonian diffeomorphisms generated by Hamiltonians $H:W\to \R$ whose support is contained in $W\setminus \p W$.
\end{enumerate}

\section{Real Liouville domains}
The examples in Section \ref{sec: A_k} and \ref{sec: chyper} are real Liouville domains whose boundary has a periodic Reeb flow. We review the notion of real Liouville domains and investigate the relation between Maslov indices of a periodic Reeb orbit of the minimal common period and its half Reeb chord on the contact boundary, see Proposition \ref{lem: indrel}.
\subsection{Definition}\label{def:real}

Let $(W, \lda)$ be a Liouville domain. An \textit{exact anti-symplectic involution} is a diffeomorphism $\rho\in \Diff(W)$ such that $\rho^2=\id$ and $\rho^*\lambda=-\lambda$. A \textit{real Liouville domain} is a triple $(W,\lambda,\rho)$ such that $(W,\lambda)$ is a Liouville domain and $\rho$ is an exact anti-symplectic involution. We denote by $\Fix(\rho)$ the fixed point set of $\rho$, and $\Fix(\rho)$ is a Lagrangian in $W$. We call $\Fix(\rho)$ a \textit{real Lagrangian} with respect to $\rho$. 

Analogously, a {\it real contact manifold} is a triple $(\Sigma,\alpha,\rho_\Sigma)$ such that $(\Sigma,\alpha)$ is a contact manifold and $\rho_\Sigma\in \Diff(\Sigma)$ is an {\it anti-contact involution}, i.e., $\rho_\Sigma^2=\id$ and $\rho_\Sigma^*\alpha=-\alpha$. The fixed point set $\Fix(\rho_\Sigma)$ forms a {\it real Legendrian}. Since $\rho^*R_\alpha=-R_\alpha$ where $R_\alpha$ is the Reeb vector field of $\alpha$, we obtain a useful identity
\begin{equation}\label{eq: anticont}
	\rho \circ Fl_t^{R_\alpha}\circ \rho= Fl_{-t}^{R_\alpha}.
\end{equation}
The boundary of a real Liouville domain $(W,\lambda,\rho)$ is a real contact manifold $(\p W,\lambda|_{\p W},\rho|_{\p W})$. By abuse of notation, we denote the restriction $\rho|_{\Sigma}$ also by $\rho$. Real Lagrangians provide admissible Lagrangians in real Liouville domains as follows.
\begin{Lem}
Let $(W,\lambda,\rho)$ be a real Liouville domain such that $L=\Fix(\rho)$ and $\p L=\Fix(\rho|_{\p W})$ are not empty. Then the real Lagrangian $L$ is admissible.
\end{Lem}
\begin{proof}
By the condition that $\rho^* \lambda = -\lambda$, it follows $\lambda|_L$ is identically zero. In particular $L$ is exact, and the intersection $L \cap \partial W$ is a Legendrian. Now denote the Liouville vector field of $\lambda$ by $X_{\lda}$ as before, i.e., $\iota_{X_{\lda}} \ow = \lda$. Then we have $\iota_{T\rho X_{\lda}} \ow = \rho^* \iota_{X_{\lda}} \ow \circ T\rho = \rho^* \lda \circ T\rho = \lda$, so $T\rho X_{\lda} = X_{\lda}$. We conclude that $X_{\lda}$ is tangent to $L = \Fix(\rho)$. 
\end{proof}

%Denote the minimal common period by $T_\text{min}$. In other words, $Fl_{T_\text{min}}^{R_{\alpha}} = \id$ and $Fl_t^{R_{\alpha}} \neq \id$ for $0 < t < T_\text{min}$. 

\subsection{Half Reeb chords of periodic Reeb orbits}\label{sec: halfreeb}
Let $(\Sigma,\alpha,\rho)$ be a real contact manifold with periodic Reeb flow and suppose that $\mathcal{L}=\Fix(\rho)$ is a nonempty, {\it path-connected} Legendrian. We assume that the Maslov class $\mu_\mathcal{L}$ vanishes. Let $\gamma: [0, T] \rightarrow \Sigma$ be periodic Reeb orbit of period $T$, starting at a point in the Legendrian $\mathcal{L}$. From \eqref{eq: anticont}, we see that $c(\frac{T}{2}) \in \mathcal{L}$. We define a corresponding Reeb chord $c: [0, \frac{T}{2}] \rightarrow \Sigma$ by
$$
c(t): = \gamma|_{[0,\frac{T}{2}]}(t).
$$
 We call this chord the \textit{half Reeb chord} associated to $\gamma$. Suppose now that both $\gamma$ and $c$ are contractible. We recall the definition of the Maslov index $\mu(\gamma)$ of a periodic Reeb orbit $\gamma$. Choose a capping disk $u:D^2\to \Sigma$ of $\gamma$ such that $u(e^{2\pi it/T})=x(t)$ and a symplectic trivialization of $u^*\xi$
$$
\phi_u:(u^*\xi,d\alpha) \longrightarrow D^2\times (\C^{n-1},\ow_\text{std}).
$$  Writing
$$
\Psi_\gamma:[0,T]\to \Sp(2n-2),\quad \Psi_\gamma(t):=\phi_u\circ TFl_t^{R_\alpha}|_\xi(\gamma(0))\circ \phi_u^{-1}
$$
for the linearization of the Reeb flow along $\gamma$ with respect to the trivialization $\phi_u$, we define the {\it Maslov index} of a periodic Reeb orbit $\gamma$ as the Robbin-Salamon index
$$
\mu(\gamma):=\mu_{RS}(\text{Gr}(\Psi_\gamma),\Delta),
$$
where $\text{Gr}(\Psi_\gamma)=\{(x,\Psi_\gamma(x)) \ |\ x\in \R^{2n}\}$ is a graph Lagrangian and $\Delta=\text{Gr}(\id)$ is a diagonal Lagrangian in $(\R^{2n}\times \R^{2n},(-\ow_\text{std})\oplus \ow_\text{std})$. Since $c_1(\xi)$ vanishes on $\pi_2(\Sigma)$ under our assumption, $\mu(\gamma)$ does not depend of the choice of capping disks. 

We examine the relation between the indices of a periodic Reeb orbit of the minimal common period and its half Reeb chord. We first need the following lemma.
 
\begin{Lem}\label{lem: hor}
Let $\Psi:[0,T]\to \Sp(2n-2)$ satisfying $\Psi(0)=\Psi(T)=\id$ and let $\Lambda$ be a Lagrangian subspace in $\C^{n-1}$. Then we have
$$
\mu_{RS}(\Psi\Lambda,\Lambda)=\mu_{RS}(\text{Gr}(\Psi),\Delta).
$$
\end{Lem}
\begin{proof}

By \cite[Theorem 3.2]{RS}, we know that $\mu_{RS}(\Psi\Lambda,\Lambda)=\mu_{RS}(\text{Gr}(\Psi),\Lambda\times \Lambda)$. The index difference 
$$
s(\Lambda\times \Lambda,\Delta;\Delta,\Delta)=\mu_{RS}(\text{Gr}(\Psi),\Delta)-\mu_{RS}(\text{Gr}(\Psi),\Lambda\times \Lambda)
$$
is given by the H\"ormander index \cite[Theorem 3.5]{RS}, and by \cite[Formula (2.12)]{Po} it can be computed as
$$
s(V_0,V_1;V'_0,V'_1)=\frac{1}{2}(\tau_{KH}(V_0,V_1,V'_0)-\tau_{KH}(V_0,V_1,V'_1)),
$$
where $V_0,V_1,V'_0,V'_1$ are Lagrangian subspaces in $\C^{n-1}$ and $\tau_{KH}$ denotes the H\"{o}rmander-Kashiwara index defined in \cite[Section 2.3]{Po}. From this formula, we conclude that $s(\Lambda\times \Lambda,\Delta;\Delta,\Delta)=0$.
\end{proof}

Let $c:[0,\frac{T}{2}]\to \Sigma$ be a Reeb chord and $\overline{c}(t):=\rho(c(\frac{T}{2}-t))$ the reflected Reeb chord. By concatenation, we obtain a  periodic Reeb orbit $\gamma=c\# \overline{c}:[0,T]\to \Sigma$. Such a periodic Reeb orbit is called {\it symmetric}. Let $v:D^+\to \Sigma$ be a capping half-disk of $c$. We define a capping half-disk of $\overline{c}$ by 
$$
\overline{v}:D^+\to \Sigma,\quad \overline{v}(z):=\rho(v(-\overline{z})).
$$
We then get a capping disk of a periodic Reeb orbit $\gamma$ by gluing $v$ and $\overline{v}$ along $v(D^+_\R)=\overline{v}(D^+_\R)$, 
$$
u:=v\#\overline{v}:D^2\to \Sigma,\quad u(z):=\begin{cases}
	v(z) & \text{for $\im z\ge 0$} \\
	\overline{v}(-z) & \text{for $\im z\le 0$}
\end{cases}.
$$
It satisfies $\rho\circ u=u\circ \mathcal{I}_{D^2}$, where $\mathcal{I}_{D^2}(z)=\overline{z}$ denote a complex conjugate map on $D^2\subset \C$. By \cite[Lemma 3.10]{FK}, there exists a symplectic trivialization
$$
\phi_u:(u^*\xi,d\alpha) \longrightarrow D^2\times (\C^{n-1},\ow_\text{std})
$$
such that $\phi_u\circ T\rho|_{\xi}(u(z))=\mathcal{I}_{\C^{n-1}}\circ \phi_u(\rho(u(z)))$ for $z\in D^+$. Such a trivialization $\phi_u$ is called \emph{symmetric} symplectic trivialization. Since it holds that $\phi_u(T_{u(z)}\mathcal{L})=\Lambda_\text{hor}^{n-1}$ for $z\in D^2\cap \{\im z=0\}$, a symmetric trivialization $\phi_v$ restricts to adapted symplectic trivializations $\phi_v$ and $\phi_{\overline{v}}$ of $v^*\xi$ and $\overline{v}^*\xi$, respectively.

\begin{prop} \label{lem: indrel}
Let $\gamma: [0, T] \rightarrow \Sigma$ be a periodic Reeb orbit of the minimal common period $T$, starting at a point in the Legendrian $\mathcal{L}$ and $c$ its half Reeb chord. Then we have $\mu(\gamma) = 2 \mu(c)$.
\end{prop}

\begin{proof}
%We first claim that $\mu(c)+\mu(\overline{c})=\mu_{RS}(\phi_v\circ TFl_t^{R_\alpha}|_\xi(\gamma(0))\circ \phi_v^{-1} \Lambda_\text{hor}^{n-1},\Lambda_\text{hor}^{n-1})$.
Consider paths of Lagrangian subspaces
$$
\Psi_c:[0,\frac{T}{2}]\to\mathcal{L}(n), \quad \Psi_c(t)=\phi_v\circ TFl_t^{R_\alpha}|_\xi(T_{c(0)}\mathcal{L})
$$
$$
\Psi_{\overline{c}}:[0,\frac{T}{2}]\to\mathcal{L}(n), \quad \Psi_{\overline{c}}(t)=\phi_{\overline{v}}\circ TFl_t^{R_\alpha}|_\xi(T_{\overline{c}(0)}\mathcal{L}).
$$
We then see that
$$
\mu(c)=\mu_{RS}(\Psi_c,\Lambda_\text{hor}^{n-1}),\quad \mu(\overline{c})=\mu_{RS}(\Psi_{\overline{c}},\Lambda_\text{hor}^{n-1}).
$$
Since $\Psi_c(\frac{T}{2})=\Psi_{\overline{c}}(0)$, we obtain $\Psi_c\#\Psi_{\overline{c}}:[0,T]\to \mathcal{L}(n)$ by concatenation and one checks that
$$
\Psi_c\#\Psi_{\overline{c}}=\phi_v\circ TFl_t^{R_\alpha}|_\xi(T_{\gamma(0)}\mathcal{L})=\phi_v\circ TFl_t^{R_\alpha}|_\xi\circ \phi_v^{-1}(\Lambda_\text{hor}^{n-1}).
$$
By catenation property of Robbin-Salamon index \cite[Theorem 2.3]{RS}, we obtain
$$
\mu(c)+\mu(\overline{c})=\mu_{RS}(\phi_v\circ TFl_t^{R_\alpha}|_\xi\circ \phi_v^{-1} \Lambda_\text{hor}^{n-1},\Lambda_\text{hor}^{n-1}).
$$
Abbreviating
$$
\Psi_\gamma(t): [0, T] \longrightarrow \Sp(2n-2), \quad t \longmapsto \phi_v \circ TFl_t^{R_{\alpha}}|_{\xi} \circ \phi_v^{-1},
$$
we see that $\text{Gr}(\Psi_\gamma(t))=\Delta$ for $t=0,T$ as $Fl_T^{R_\alpha}=\id$. By Lemma \ref{lem: hor}, we have
$$
\mu_{RS}(\Psi_\gamma \Lambda_\text{hor}^{n-1},\Lambda_\text{hor}^{n-1})=\mu_{RS}(\text{Gr}(\Psi_\gamma),\Delta)=\mu(\gamma).
$$
Since $c$ and $\overline{c}$ lie in the same Morse-Bott component $\mathcal{L}_{\frac{T}{2}}$, we get $\mu(c)=\mu(\overline{c})$.
\end{proof}

\subsection{A criterion for Morse-Bott type Reeb chords}

In this section, we give a criterion to check Reeb chords are of Morse-Bott type where the contact form is already of Morse-Bott type.
 
\begin{Def}\label{def: mbreebchord}
A contact form $\alpha$ on $\Sigma$ is of {\it Morse-Bott type} if the following holds.
\begin{enumerate}[label=(\arabic*)]
	\item The spectrum of periodic Reeb orbits $\Spec(\Sigma,\alpha)$ is discrete.
	\item For $T\in \Spec(\Sigma,\alpha)$, the set $\mathcal{N}_T:=\{z\in \Sigma\ |\ Fl_T^{R_\alpha}(z)=z\}$ is a (possibly disconnected) closed submanifold of $\Sigma$ such that $T_z\mathcal{N}_T=\ker(T_zFl_T^{R_\alpha}-\id)$ for $z\in \mathcal{N}_T$.
\end{enumerate}	
Here, $\Spec(\Sigma,\alpha)$ denotes the set of all periods of periodic Reeb orbits.
\end{Def}
\begin{Rem}
We do not assume that rank $d\alpha|_{\mathcal{N}_T}$ is locally constant, which is a part of definition in \cite[Definition 1.7]{Bou}.
\end{Rem}
Let $(\Sigma,\alpha,\rho)$ be a real contact manifold with a Morse-Bott type contact form $\alpha$. Since each Reeb chord gives rise to a symmetric periodic Reeb orbit as in Section \ref{sec: halfreeb}, we deduce that
$$
\Spec(\Sigma,\alpha,\mathcal{L})\subset \frac{1}{2}\Spec(\Sigma,\alpha).
$$
In particular, $\Spec(\Sigma,\alpha,\mathcal{L})$ is discrete. Using the identity \eqref{eq: anticont}, one can check that
\begin{equation}\label{eq: mbc}
\mathcal{L}_T=\mathcal{N}_{2T}\cap \mathcal{L}\quad \text{for $T\in \Spec(\Sigma,\alpha,\mathcal{L})$}.
\end{equation}

We now derive a criterion for Morse-Bott type Reeb chords.
\begin{prop}\label{prop: mbcrit}
	Let $(\Sigma,\alpha,\rho)$ be a closed real contact manifold with a Morse-Bott type contact form $\alpha$ and let $\mathcal{L}=\Fix(\rho)$ be a nonempty Legendrian. Suppose that $\mathcal{N}_{2T}$ and $\mathcal{L}$ intersect cleanly along a closed submanifold $\mathcal{L}_T$ for all $T\in \Spec(\Sigma,\alpha,\mathcal{L})$. Then, Reeb chords are of Morse-Bott type.
\end{prop}
\begin{proof}
Fix a Morse-Bott component $\mathcal{L}_T$. For simplicity, we use the following notation
$$
\Phi_t=TFl_t^{R_\alpha},\quad \varrho=T\rho.
$$
From the above discussion, it remains to show that $Fl_{-T}^{R_\alpha}(\mathcal{L})$ and $\mathcal{L}$ intersect cleanly along $\mathcal{L}_T$, i.e., 
$$
T_z\mathcal{L}_T=\Phi_{-T}(T_w\mathcal{L})\cap T_z\mathcal{L}\quad \text{for $z\in \mathcal{L}_T$ and $w=Fl_T^{R_\alpha}(z)$}.
$$
By \eqref{eq: mbcp}, we see that $T_z\mathcal{L}_T\subset \Phi_{-T}(T_w\mathcal{L})\cap T_z\mathcal{L}$. This means that $\Phi_TX\in T\mathcal{L}$ for $X\in T\mathcal{L}_T$. Note that $T\mathcal{L}=\ker(\varrho-\id)$ and the identity \eqref{eq: anticont} implies that $\varrho\circ \Phi_t\circ \varrho=\Phi_{-t}$ for all $t\in \R$.
We shall show that $T_z\mathcal{L}_T\supset \Phi_{-T}(T_w\mathcal{L})\cap T_z\mathcal{L}$. We check that for $X\in\ \Phi_{-T}(T_w\mathcal{L})\cap T_z\mathcal{L}$,
$$
\Phi_{2T}X=\Phi_T\circ \Phi_T X=\varrho \circ \Phi_{-T}\circ \varrho \circ \Phi_T X=\varrho\circ \Phi_{-T} \circ\Phi_TX=\varrho X=X.
$$
This shows that $X\in T_z\mathcal{N}_{2T}\cap T_z\mathcal{L}=T_z\mathcal{L}_T$.
\end{proof}
We obtain an immediate corollary that will be used in Section \ref{sec: A_k} and \ref{sec: chyper}. 
\begin{Cor}\label{cor: mbtype}
Let $(\Sigma,\alpha,\rho)$ and $\mathcal{L}$ be as in Proposition \ref{prop: mbcrit}. If $\mathcal{N}_{2T}=\Sigma$ holds for all $T\in \Spec(\Sigma,\alpha,\mathcal{L})$, then Reeb chords are of Morse-Bott type.
\end{Cor}

\section{Fibered twists} \label{Fibered twist}
In this section, we define a special class of symplectomorphisms, called \emph{fibered twists}. Let $(W, \lambda)$ be a Liouville domain whose boundary admits a periodic Reeb flow, i.e., $Fl_T^{R_{\alpha}}= \id$ for some $T > 0$. For simplicity, we assume that the minimal period of the Reeb flow is $1$. Consider a Hamiltonian $H_\tau: \widehat W \rightarrow \R$ such that
\begin{itemize}
	\item $H_\tau$ vanishes on $W \setminus \left( [1-\epsilon, 1] \times \partial W \right)$, and
	\item $H_\tau(r, x)=h(r)$ on $[1-\epsilon, +\infty) \times \partial W$ where $h'(r)$, $h''(r) >0$ for all $r \in (1-\epsilon, 1)$ and $h'(r)=1$ for all $r \in [1, \infty)$. 
\end{itemize}
\begin{Def}
The time 1-map of the Hamiltonian diffeomorphism generated by $H_\tau$ is called a \textit{fibered twist}, denoted by $\tau: \widehat{W} \rightarrow \widehat{W}$. 
\end{Def}
By the construction, $\tau$ is the identity on $W \setminus \left([1-\epsilon, 1] \times \partial W \right)$ and $[1, +\infty) \times \partial W$. In particular, we have $\tau\in \symp^c(\widehat{W})$. A priori the definition of fibered twists depends on the choice of a Hamiltonian $H_\tau$. However, its connected component $[\tau]$ in $\pi_0(\symp^c(\widehat W))$ is well-defined. This can be shown by interpolating defining Hamiltonians.

Fibered twists have played an important role in a study of symplectic mapping class groups, see \cite{Sei}, \cite{CDvK} and \cite{Igo}. For example, $[\tau]$ has infinite order in $\pi_0(\symp^c(T^*S^n))$, whereas its order in $\pi_0(\Diff^c(T^*S^n))$ is finite for $n$ is even, see \cite[Corollary 4.5]{Sei}, \cite[Proposition 2.23]{FS05}. In \cite{FS05}, the fact that $[\tau]$ has infinite order in $\pi_0(\symp^c(T^*S^n))$ actually relies on existence of a Lagrangian, namely, a cotangent fiber. Theorem \ref{ThmB} can be regarded as a  generalized version in terms of the wrapped Floer homology. 

\section{Proof of main theorems}
\subsection{Theorem \ref{ThmA}} \label{proof}
We first need the following lemma. In the below, we denote by $D^n_\delta=\{p\in \R^n\ |\ |p|\le \delta\}$ the closed disk of radius $\delta$ in $\R^n$.

\begin{Lem} \label{family of Lag}
Let $(W, \lambda)$ be a Liouville domain and $L$ be  an admissible Lagrangian ball in $W$. Then there exists a $D_\delta^n$-family of pairwise disjoint Lagrangian balls $\{L_z\}_{z \in D^n_{\delta}}$ such that each Lagrangian $L_z$ is admissible and is Hamiltonian isotopic to $L$.
\end{Lem}
Note that this obviously holds for cotangent fibers in a disk cotangent bundle over a smooth manifold. A proof follows from a suitable Lagrangian neighborhood theorem which is a typical application of the Moser trick:
\begin{Lem} \label{neighborhood of Lag}
Let $(W^{2n}, \lambda)$ be a Liouville domain and let $L$ be an admissible Lagrangian ball in $W$. Then there exist a tubular neighborhood $\nu_W (L)$ of $L$ in $W$ and a symplectomorphism
$$
\Phi: (D_\delta^n \times D^n, dp \wedge dq) \longrightarrow (\nu_W (L), d\lambda|_{\nu_W(L)})
$$
such that $\Phi(\{0\} \times D^n)=L$ and $\Phi^*\lambda=pdq$ near $\p W$. 
\end{Lem}
\begin{proof}
Consider a symplectic manifold $(\mathbb{R}^n \times D^n, dp \wedge dq=\sum_jdp_j\wedge dq_j)$ where $(q_1,\cdots,q_n,p_1,\cdots,p_n)$ are  coordinates on $\mathbb{R}^n\times D^n$. Then $(\mathbb{R}^n \times S^{n-1}, pdq)$ is a contact manifold and $\{0\} \times S^{n-1}$ is a Legendrian sphere. Since $\partial L$ is a Legendrian sphere of a contact manifold $(\p W, \alpha=\lambda|_{\p W})$, there exists $\delta_0>0$ and a contactomorphism
$$\Psi: D^n_{\delta_0} \times S^{n-1} \rightarrow \nu_{\p W}(\partial L)$$
such that $\Psi^*\alpha=pdq$, see \cite[Theorem 6.2.2]{Gei}. We extend this map to each tubular neighborhood in the following way:
$$\tilde{\Psi}: D^n_{\delta_0} \times [1-\epsilon, 1] \times S^{n-1} \longrightarrow [1-\epsilon, 1] \times \nu_{\p W}(\partial L), \quad (q, r, p) \longmapsto (r, \Psi(q, p)).$$
We denote $A^n(1-\epsilon,1):=\{p \in D^n \ | \ 1-\epsilon \le |p| \le 1\}$ and identify $A^n(1-\epsilon,1)$ with $[1-\epsilon,1]\times S^{n-1}$ via the map $p \mapsto (|p|,\frac{p}{|p|})$. Under this identification it satisfies $\tilde{\Psi}(r\alpha)=pdq$. Let $g_0$ be a standard Euclidean metric on $\mathbb{R}^n \times D^n$. This defines a metric $g=\tilde{\Psi}_* g_0$ on $[1-\epsilon, 1] \times \nu_{\p W}(\partial L)$. Let $J$ be the almost complex structure on $[1-\epsilon, 1] \times \nu_{\p W}(\partial L)$ corresponding to $g$. Note that $J=\tilde{\Psi}_* J_0$ where $J_0$ is the standard complex structure on $\mathbb{R}^n \times D^n$. We extend $J$ to a almost complex structure $\tilde{J}$ on a neighborhood $\nu_W (L)$ of $L$. We take a diffeomorphism
$$\phi: D^n \rightarrow L$$
satisfying $\phi(p)=\tilde{\Psi}(0, |p|, {p \over |p|})$ for all $p$ with $|p|\ge 1-\epsilon$. We define a bundle map $I_g: \mathbb{R}^n \times D^n \rightarrow TL$ over $\phi$ such that the diagram
\begin{gather*}
\xymatrix
{
\mathbb{R}^n \times D^n \ar[rr]^{I_g} \ar[d] & & TL \ar[d]
\\ 
D^n \ar[rr]^{\phi} & & L
}
\end{gather*}
commutes and $g_{\phi(p)}(I_g(p)[q], v)=\left<q, (T_p\phi)^{-1}(v)\right>$ for all $p \in D^n$ and $v \in T_{\phi(p)} L$. We construct a map
$$
\tilde{\phi}: D_\delta^n \times D^n \longrightarrow \nu_W (L), \quad (q, p) \longmapsto \exp_{\phi(p)}\left(\tilde{J}I_g(p) [ q ] \right)
$$ 
where the exponential corresponds to the metric induced by the almost complex structure $\tilde{J}$. This map is diffeomorphism by choosing sufficiently small $0<\delta \le \delta_0$ and sufficiently small neighborhood $\nu_W (L)$ of $L$. Note that the differential of $\tilde{\phi}$ at $(0, p)$ is given by
$$d\tilde{\phi}(0, p) \left[ a{\partial \over \partial q}+b{\partial \over \partial p} \right]=d\phi(p)\left[ b{\partial \over \partial p} \right]+\tilde{J} I_g(p) \left[ a \right]$$
It follows that
\begin{eqnarray*}
&& \tilde{\phi}^* \omega_{(0, p)} \left( a_1{\partial \over \partial q}+b_1{\partial \over \partial p}, a_2{\partial \over \partial q}+b_2{\partial \over \partial p} \right) \\
&=& \omega_{\phi(p)}\left( d\phi(p)\left[ b_1{\partial \over \partial p} \right]+\tilde{J}I_g (p)\left[ a_1 \right], d\phi(p)\left[ b_2{\partial \over \partial p} \right]+\tilde{J}I_g(p) \left[ a_2 \right]\right) \\
&=& \omega_{\phi(p)}\left( d\phi(p)\left[ b_1{\partial \over \partial p} \right], \tilde{J}I_g(p) \left[ a_2 \right] \right)-\omega_{\phi(p)}\left( d\phi(p)\left[ b_2{\partial \over \partial p} \right], \tilde{J}I_g(p) \left[ a_1 \right] \right) \\
&=& g_{\phi(p)}\left( d\phi(p)\left[ b_1{\partial \over \partial p} \right], I_g(p) \left[ a_2 \right] \right)-g_{\phi(p)}\left( d\phi(p)\left[ b_2{\partial \over \partial p} \right], I_g(p) \left[ a_1 \right] \right) \\
&=& \left<a_2, b_1 \right>-\left<a_1, b_2 \right>=dp \wedge dq \left( a_1{\partial \over \partial q}+b_1{\partial \over \partial p}, a_2{\partial \over \partial q}+b_2{\partial \over \partial p} \right).
\end{eqnarray*}
Thus $\tilde{\phi}^* \omega$ and $dp \wedge dq$ agree at the zero section $\{0\} \times D^n$. Since $\tilde{\phi}^* \omega-dp \wedge dq=d(\tilde{\phi}^* \lambda-pdq)$, we can apply Moser trick to $\omega_t=t \tilde{\phi}^* \omega+ (1-t) dp \wedge dq$. We want to have a family of diffeomorphism $\rho_t : D^n_\delta \times D^n  \rightarrow D^n_\delta \times D^n$ such that $\rho_t^* \omega_t=\omega_0=dp \wedge dq$. It is enough to find a family of vector field $v_t$ generating $\rho_t$ and this condition is written as
$$0={d \over dt}(\rho_t^* \omega_t)=\rho_t^* (L_{v_t}\omega_t+{d \over dt}\omega_t).$$
This is satisfied if the identity $\iota_{v_t} \omega_t=pdq-\tilde{\phi}^* \lambda$ holds. Since $pdq-\tilde{\phi}^* \lambda=0$ near the boundary, we have $v_t=0$ near the boundary and so $\rho_t$ is identity near the boundary. Thus $\Phi=\tilde{\phi} \circ \rho_1$ is a desired map.
\end{proof}

We are ready to prove Lemma \ref{family of Lag}.
\begin{proof}[Proof of Lemma \ref{family of Lag}]
By Lemma \ref{neighborhood of Lag}, we have a symplectomorphism
$$
\Phi: (D_\delta^n \times D^n, dp \wedge dq) \rightarrow (\nu_W (L), d\lambda|_{\nu_W(L)})
$$
satisfying $\Phi(\{0\} \times D^n)=L$ and $\Phi^*\lambda=pdq$ near $\p W$. Consider a family of Lagrangian $\{\{ z \} \times D^n\}_{z\in D_\delta^n}$ in $(D_\delta^n \times D^n, dp \wedge dq)$. Then, each Lagrangian $\{z\}\times D^n$ is Hamiltonian isotopic to $\{0\} \times D^n$ by the Hamiltonian flow of $H_z(q, p)=z \cdot p$. We define a family of Lagrangians by
$$
L_z :=\Phi(\{ z \} \times D^n)\quad \text{for $z\in D_{\delta}^n$}.
$$
Each Lagrangian $L_z$ is Hamiltonian isotopic to $L_0$ in $W$ by considering a suitable cut-off function with a neighborhood of $\nu_W (L)$. Since $\Phi$ preserves primitive 1-forms near the boundary and the boundary of $\{z\} \times D^n$ is a Legendrian sphere with tangential Liouville vector field, the boundary of $L_z$ is also a Legendrian sphere with tangential Liouville vector field near the boundary for every $z \in D^n_\delta$.
\end{proof}
We define the notion of a linear growth of the wrapped Floer homology.
\begin{Def}\label{def: lineargrowth}
For an admissible Lagrangian $L$ in a Liouville domain $(W,\lambda)$, we say that its wrapped Floer homology $\WFH_*(L;W)$ has a \textit{linear growth} if the limit
\begin{equation}\label{eq: growth}
\liminf_{c\rightarrow \infty} {\dim \WFH^{< c}_*(L; W) \over c}
\end{equation}\label{growthrateintro}
is positive.
\end{Def}
\begin{Rem}\
\begin{itemize}
	\item The limit \eqref{eq: growth} measures a growth of the dimension of $\WFH_*(L;W)$ along the action filtration. 
	\item If the contact boundary $(\p W,\lambda|_{\p W})$ admits a periodic Reeb flow, then the limit \eqref{eq: growth} is always finite by Theorem \ref{thm: ss}
	\item If $\WFH_*(L;W)$ has a linear growth, then we  obtain $\dim \WFH_*(L;W)=\infty$.
\end{itemize}
\end{Rem}

We now give a proof of Theorem \ref{ThmA}.

\begin{proof}[Proof of Theorem \ref{ThmA}]
Without loss of generality, we may assume that the minimal period of the Reeb flow is 1. The condition $H^1_c(W;\R)\cong H^1_c(\widehat{W}; \mathbb{R})=0$ imposes that $\text{Ham}^c(\widehat{W})=\symp_0^c(\widehat{W})$. Let $\phi$ be a compactly supported symplectomorphism satisfying $[\phi]=[\tau^k] \in \pi_0(\symp^c(\widehat{W}))$ for some $k \ne 0$. We may assume that $\phi$ has its support in $W$ by rescaling, so we have $\tau^{mk}\phi^{-m}\in \text{Ham}^c(W)$. Let $\{ L_z \}_{z \in D^n_\delta}$ be a family of pairwise disjoint Lagrangians which are Hamiltonian isotopic to each other as in Lemma \ref{family of Lag}. We verify that
\begin{eqnarray*}
\LFH_*(L_z, \phi^m(L)) &\cong& \LFH_*(L,\phi^m(L)) \quad \text{by Lemma \ref{lem: secondinv}}\\
&\cong& \LFH_* (L, \tau^{mk}(L)) \quad \text{by Lemma \ref{lem: firstinv}}. 
\end{eqnarray*}
Choose $\epsilon>0$ small enough so that $l+\epsilon\notin \text{Spec}(\p W,\alpha,\p L)$ for all $l\in \N\cup \{0\}$. Combining with the properties \ref{p1} and \ref{p2}, we obtain an isomorphism
$$
\LFH_*(L_z,\phi^m(L))\cong \WFH_*^{<mk+\epsilon}(L;W)
$$
for all $m\in \N$ and $z\in D_\delta^n$. Let
$$
b:=\liminf_{c\rightarrow \infty} {\dim \WFH^{< c}_*(L; W) \over c}>0.
$$
Then, we have
$$
\dim \WFH_*^{mk+\epsilon}(L;W)\ge \frac{b}{2} (mk+\epsilon)
$$
for each $m\in \N$ by choosing a subsequence if necessary. We assume that $L_z$ and $\phi^m(L)$ intersect transversally. By the property \ref{p3}, the manifold with boundary $\phi^m(L)$ intersects $L_z$ at least $\lfloor {b \over 2} (mk+\epsilon) \rfloor$-times for every $z \in D^n_{\delta}$ and $m \in \mathbb{N}$. We consider the standard metric $g_0$ on $D_\delta^n \times D^n$. Its push-forward $\Phi_* g_0$ defines a metric on a neighborhood $\nu_W (L)$ of $L$. This extends to a metric $g$ on $W$. With this metric, the $n$-dimensional measure $\mu_g(\phi^m(L))$ of $\phi^m(L)$ is greater than $\lfloor {b \over 2} (mk+\epsilon) \rfloor \cdot \mu_{g_0}(D_\delta^n)$. Therefore, we have that
$$\mu_g(\phi^m(L)) \ge {b \over 2} (mk+\epsilon) \Delta -\Delta$$
where $\Delta=\mu_{g_0}(D_\delta^n)>0$. 

If $L_z$ and $\phi^m(L)$ are not transverse, then we follow the idea in \cite[Section 2.3]{FS05}. Choose a sequence $\phi_i\in \symp^c(W)$ such that
\begin{itemize}
	\item $\phi_i^m(L)$ and $L_z$ are transverse for all $i$, and
	\item $\phi_i\to \phi$ in the $C^\infty$-topology.
\end{itemize}
We then verify that $[\phi_i]=[\phi]\in \pi_0(\symp^c(W))$ for large $i$, and
$$
\mu_g(\phi^m(L))=\lim_{i\to \infty}\mu_g(\phi_i^m(L))\ge \frac{b}{2}(mk+\epsilon)\Delta-\Delta.
$$
Consider a smooth embedding $\sigma: Q^n \rightarrow \widehat{W}$ such that $\sigma(Q^n)$ contains the Lagrangian ball $L$. By definition of the slow volume growth $s_n$, we have 
$$s_n(\phi) \ge \liminf_{m \rightarrow +\infty}{\log \mu_g(\phi^m (\sigma)) \over \log m} \ge \liminf_{m\rightarrow +\infty} {\log \mu_g(\phi^m (L)) \over \log m}.$$
Since $b>0$, the $n$-dimensional slow volume growth $s_n(\phi)$ satisfies the inequality
$$\liminf_{m\rightarrow +\infty} {\log \mu_g(\phi^m (L)) \over \log m} \ge \liminf_{m\rightarrow +\infty} {\log ({b \over 2} (mk+\epsilon) \Delta -\Delta) \over \log m} =1.$$
This completes the proof of Theorem \ref{main thm}.
\end{proof}

\subsection{Theorem \ref{ThmB}}
In fact, we prove the following refined statement. Theorem \ref{ThmB} follows immediately.
\begin{Thm} \label{WFH and fib}
Let $(W,\lambda)$ be a Liouville domain with a periodic Reeb flow on the boundary and assume that $H^1_c(W,\R)=0$. If a class $[\tau]$ of a fibered twist $\tau:\widehat{W}\to \widehat{W}$ has a finite order in $\pi_0(\symp^c(\widehat{W}))$, then we have
$$
\dim \WFH_* (L; W) \le \dim H_*(L,\p L)<\infty
$$
for any admissible Lagrangian $L$ in $W$.
\end{Thm}
\begin{proof}
As in the proof of Theorem \ref{ThmA}, we know $\text{Ham}^c(\widehat{W})=\symp_0^c(\widehat{W})$ due to $H^1_c(W;\R)=0$.
Without loss of generality, we may assume that $[\tau]=[\id]\in \pi_0(\symp^c(\widehat{W}))$. Choose an admissible Hamiltonian $D: \widehat{W} \rightarrow \mathbb{R}$ of slope $\epsilon<\text{min} (\text{Spec}(\p W,\alpha, \p L))$ such that $D$ is $C^2$-small on $W$ and $D(r,x)=d(r)$ on $[1-\epsilon,1]\times \p W$ for some smooth function $d$. Let $H_N=N\cdot H_\tau+D$ for $N\in \N\cup\{0\}$, where $H_\tau$ is a Hamiltonian that generates a fibered twist $\tau$, see Section \ref{Fibered twist}. One can check that $$Fl_1^{X_{H_N}} = Fl_1^{X_{N \cdot H}} \circ Fl_1^{X_D}=\tau^N\circ Fl_1^{X_D}.$$ 
By \ref{p2} and Floer \cite{Floer}, we obtain
$$
\WFH_*(L;H_0) \cong \LFH_*(L, Fl_1^{X_D}(L)) \cong H_*(L,\p L),
$$
possibly up to a grading shift. From now on, all isomorphisms below will be understood up to grading shift. We now verify that for $N\ge 0$,
\begin{eqnarray*}
\WFH_*(L; H_0) &\cong& \LFH_*(L, Fl_1^{X_D}(L)) \ \quad\quad\quad \text{by \ref{p2}}\\
&\cong& \LFH_*(L, \tau^N \circ Fl_1^{X_D}(L))	\quad \text{by \ref{p4}}\\
&\cong& \LFH_*(L, Fl^{X_{H_N}}_1(L)) \\
&\cong& \WFH_*(L; H_N) \quad\quad\quad\qquad \text{by \ref{p2}}.
\end{eqnarray*}
Hence, we see that 
$\WFH_*(L; H_N) \cong H_*(L,\p L)$
for all $N \in \mathbb{N}$. This fits into the following sequence of maps
$$
\xymatrix{
\WFH_*(L; H_0) \ar[r]^{\Phi_{01}} \ar[d]^{\cong} & \WFH_*(L; H_1) \ar[r]^{\Phi_{12}} \ar[d]^{\cong} & \WFH_*(L; H_2) \ar[r]^{\Phi_{23}} \ar[d]^{\cong} & \WFH_*(L; H_3)  \ar[r] \ar[d]^{\cong} & \cdots
\\ H_*(L,\p L)  & H_*(L,\p L)  & H_*(L,\p L)  & H_*(L,\p L)  &
}
$$
where $\Phi_{N,N+1}$'s are the continuation maps. Since the family $\{H_N \}_{N\ge 0}$ is cofinal in the set of admissible Hamiltonians, we have
$$\WFH_*(L; W)= \lim_{\underset{N}{\longrightarrow}} {\WFH_*(L; H_N)}.$$
This implies the desired inequality $\dim \WFH_*(L; W) \le \dim H_*(L,\p L)$.
\end{proof}
A simple example which fits the above theorem is the following.
\begin{Ex}
The fibered twists on $\C^n$ with the standard Liouville structure are compactly supported symplectically isotopic to the identity. Note that $\WFH_*(\R^n; \C^n)$ vanishes, where $\R^n$ denote the real Lagrangian, see \cite[Proposition 2.9]{Iri13} for computation. In particular $\WFH_*(\R^n; \C^n)$ is finite dimensional.
\end{Ex}

%%%%%%%%%%%%%%%%%%%%%%%%%%%%%%%%%%%%%%%%%%%%%%%%%%%%%%%%%%%%%%%%%
%%%%%%%%%%%%%%%%%%%%%%% Proof of main theorem \label{proof}
%%%%%%%%%%%%%%%%%%%%%%%%%%%%%%%%%%%%%%%%%%%%%%%%%%%%%%%%%%%%%%%%%

%%%%%%%%%%%%%%%%%%%%%%%%%%%%%%%%%%%%%%%%%%%%%%%%%%%%%%%%%%%%%%%%%
%%%%%%%%%%%%%%%%%%%%%%% Examples \label{example}
%%%%%%%%%%%%%%%%%%%%%%%%%%%%%%%%%%%%%%%%%%%%%%%%%%%%%%%%%%%%%%%%%

\section{Cotangent bundles over CROSSes} \label{example}
Theorem \ref{main thm} can be applied to the cotangent bundle $(T^* B, d \lambda_\text{can})$ over a compact rank one symmetric space $B$ (shortly we write CROSS), where $\lambda_\text{can}$ denotes a canonical 1-form. The CROSSes are closed manifolds
$$
S^n, \quad \R P^n, \quad \C P^n, \quad \mathbb{H} P^n,\quad C a P^n
$$
with their canonical metrics whose geodesics are all periodic with the same period. By Wadsley's theorem, geodesics of a CROSS admit a common period and we normalized the metric such that the minimal common period is 1. We refer to \cite[Section 2.2]{FS05} for more details on CROSSes. Since all geodesics on $B$ are periodic and the Reeb flow on the unit cotangent bundle ($ST^*B,\alpha:=\lambda_\text{can}|_{ST^*B})$ is given by the geodesic flow, the Reeb flow on the boundary of the disk cotangent bundle over $B$ is periodic. Denote by $W:=\{(q, p) \in T^* B \min |p|_g \le 1\}$ the disk cotangent bundle over $B$, where $|\cdot |_g$ denotes the norm  induced by the canonical metric $g$ on $B$. Then, $(W,\lambda_\text{can})$ is a Liouville domain with a periodic Reeb flow on the boundary. Let $L$ be a disk cotangent fiber at $x\in B$ and $\p L=ST_x^*B$ a Legendrian in the unit cotangent bundle $\p W=ST^*B$. 
\begin{Lem}
On $ST^*B$, Reeb chords are of Morse-Bott type. More precisely, we have
\begin{enumerate}[label=(\arabic*)]
	\item $\Spec(\p W,\alpha,\p L)=T_0\cdot \N$ where $T_0$ is the minimal period of the Reeb flow.
	\item For $T\in \Spec(\p W,\alpha,\p L)$, a Morse-Bott component $\mathcal{L}_T$ is diffeomorphic to $\p L\cong S^{n-1}$.
\end{enumerate}
\end{Lem}
\begin{proof}
Since CROSSes have the property that all geodesics are embedded circles of equal period \cite[Section 2.2]{FS05}, every Reeb chord (starting and ending at $\p L$) is given by a multiple cover of a periodic Reeb orbits. Indeed, let $x:[0,T]\to \p W$ be a Reeb chord. We denote by $\pi:ST^*B\to B$ the canonical projection. Since a geodesic $\gamma(t):=\pi(x(t))$ is closed, i.e., $\gamma(0)=\gamma(T)$, we obtain $\gamma'(0)=\gamma'(T)$. This implies that $x(0)=x(T)$. Hence, we know that $\Spec(\p W,\alpha,\p L)=T_0\cdot \N$ where $T_0$ is the minimal period of the Reeb flow. Since $Fl_{T_0}^{R_\alpha}=\id$, the conclusion follows immediately.
\end{proof}

With this lemma, we re-prove the results in \cite[Theorem 1]{FS05} using Theorem \ref{ThmA} and \ref{ThmB}. 

\begin{Thm} \label{cot over CROSS}
Let $B$ be an $n$-dimensional CROSS and let $\tau: T^* B \rightarrow T^* B$ be a twist. 
\begin{itemize}
	\item If $\phi \in \symp^c(T^* B)$ is such that $[\phi]=[\tau^k] \in \pi_0(\symp^c(T^* B))$ for some $k \in \mathbb{Z} \setminus \{0\}$, then we have $s_n(\phi) \ge 1$.
	\item a class $[\tau]$ has an infinite order in $\pi_0(\symp^c(T^*B))$.
\end{itemize}
\end{Thm}

\begin{proof}
Since $H^1_c(W;\R) \cong H_{2n-1}(T^* B;\R) \cong H_{2n-1}(B;\R)$, the condition $H^1_c(W;\R)=0$ holds except for $B=S^1$. The case for $B = S^1$ can be shown by the topological argument given in \cite[Section 2.1]{FS05}, so we now consider $B\ne S^1$. Choose a point $x \in B$ and a Lagrangian $L=W \cap T^*_x B$, i.e., a disk cotangent fiber at $x$. Clearly, the Lagrangian $L$ is admissible and is diffeomorphic to the ball. 

We claim that the wrapped Floer homology $\WFH_*(L; W)$ is infinite dimensional. From now on, the coefficient group of singular homology is understood as $\Z_2$. By Abbondandolo-Portaluri-Schwarz \cite{APS}, there exists an isomorphism
$$
\WFH_*(L; W) \cong H_*(\Omega_x B),
$$
where $\Omega_x B$ denotes the based loop space of contractible loops in $B$ with a base point $x$. It thus suffices to show that the homology $H_*(\Omega_x B)$ is infinite dimensional.\\\\
{\it Case 1.} $\pi_1(B)=0$.

Suppose on the contrary that the homology $H_*(\Omega_x B)$ is finite dimensional. Consider the fibration
$$
\Omega_x B \rightarrow \mathcal{P}_x B \xrightarrow{ev} B$$
where $\mathcal{P}_x B$ is the space of paths starting at $x$ and $ev$ is the evaluation map given by the end point of a path. By Leray-Serre theorem for this fibration, there exists a spectral sequence $E_{pq}^r$ converging to $H_*(\mathcal{P}_xB)$ whose $E^2$-page is given by
$$
E_{pq}^2=H_p(B) \otimes H_q(\Omega_x B).
$$
In particular, we find that $E^2_{p_0q_0} \cong \Z_2$ if $p_0 = \dim B$ and $q_0$ is the largest degree such that $H_*(\Omega_x B)$ does not vanish. Since  $H_{p_0 + q_0}(\mathcal{P}_x B)\cong E^2_{p_0q_0} \cong \Z_2$, this contradicts to the fact that the path space $\mathcal{P}_x B$ is contractible. Therefore, $H_*(\Omega_x B)$ is infinite dimensional.\\\\
{\it Case 2.} $\pi_1(B)\ne 0$.

It suffices to consider $B=\mathbb{R}P^n$. In this case, we directly consider the Morse-Bott homology $H_*(\Omega_x B)$, using the energy functional
$$\mathcal{E}: \Omega_x B \longrightarrow \mathbb{R}, \quad \mathcal{E}(\gamma)=\int_{S^1} {|\dot{\gamma}(t)|^2} dt.
$$
Note that the Morse-Bott indices of periodic geodesics in $\mathbb{R}P^n$ are all zero \cite[Proposition 2.1]{FS05}, and by the periodicity of the geodesic flow, we should have that $H_*(\Omega_x B)$ is infinite dimensional.

Now, we have seen that $\WFH_*(L; W)$ is infinite dimensional. It follows from the spectral sequence \eqref{eq: MBss} that the dimension of the filtered homology $\dim \WFH_*^{< c}(L, W)$ increases linearly along $c$, i.e., the limit \eqref{eq: growth} is positive. By Theorem \ref{ThmA} and \ref{ThmB}, we complete the proof.\end{proof}

%%zero of homogeneous polynomial on C^{n+1}.
%Let us discuss another class of examples. For $d \in \mathbb{N} \setminus 2 \mathbb{N}$, we consider the following affine variety
%$$V_{n, d}=\{z=(z_0, \cdots, z_n) \in \mathbb{C} | z_0^d+z_1^d+ \cdots +z_n^d=0 \}$$
%of homogeneous polynomial $f_d(z)=z_0^d+z_1^d+ \cdots +z_n^d$. If we consider a perturbation $W_{n, d}$ near the origin of $V_{n, d} \cap B_1(0)$ with restriction of the standard 1-form $\alpha_0={i \over 2} \sum_{j=0}^{n}(z_j d\bar{z}_j -\bar{z}_j dz_j)$, then $W_{n ,d}$ defines a Liouville domain with the boundary having periodic Reeb flow. Thus we can define the fibered map $\tau$ on $\widehat{W}_{n, d}$ and $\tau \in Symp^c(\widehat{W}_{n, d})$.

%\begin{Thm} \label{link of homog poly}
%If $n \ge 2$ and an odd number $d$ satisfies $d \ne n, n+2$, then the symplectomorphism $\tau: W_{n, d} \rightarrow W_{n, d}$ is not in the identity component of $Symp^c(\widehat{W}_{n ,d})$. If a compactly supported symplectomorphism $\phi$ satisfies $[\phi]=[\tau^k] \in \pi_0(Symp^c(\widehat{W}_{n ,d}))$ for some $k \ne 0$, then $s_n(\phi) \ge 1$.
%\end{Thm}

\section{Milnor fibers of $A_k$-type singularities} \label{sec: A_k}
\subsection{Definition}
Let $f: \C^{n+1} \rightarrow \C$ be a polynomial with an isolated critical point at the origin. Consider its affine variety $f^{-1}(0)$. We define 
$$
W := \{z \in \C^{n+1}\;|\; f(z) = \epsilon \cdot \zeta(|z|)\} \cap B^{2n+2}
$$
for sufficiently small $\epsilon > 0$, where $\zeta \in C^{\infty}(\R)$ is a monotone increasing cut-off function such that $\zeta(x) = 1$ for $x \leq 1/4$ and $\zeta(x) = 0$ for $x \geq 3/4$. The domain $W$ is called a \textit{Milnor fiber}. The restriction of the standard Liouville form of $\C^{n+1}$ to $W$, denoted by $
\lda : = \lda_{\text{std}}|_W$, defines a Stein structure and hence a Liouville domain $(W, \lda)$. Its contact type boundary is called a \textit{link} of a singularity $f$.

In particular, a polynomial $f: \C^{n+1} \rightarrow \C$ is called \textit{$A_k$-type} for $k\in \N$ if it has the form
$$
f(z) = z_0^{k+1} + z_1^{2} + \cdots + z_n^2.
$$
For an $A_k$-type singularity $f$, we denote the corresponding Liouville domain and its boundary by $W_k$ and $\Sigma_ k$ respectively. The boundary $\Sigma_k$ is also called a \emph{Brieskorn manifold}.

\subsection{Periodic Reeb flow}
The $A_k$-type polynomials are weighted homogeneous of weights $(k+1, 2, \dots, 2)$. More precisely, 
$$
f(\eta^{1/(k+1)} z_0, \eta^{1/2} z_1, \dots, \eta^{1/2} z_n) = \eta \cdot f(z)
$$
for all $\eta \in \C^*$. In general, links of weighted homogeneous polynomials admit a contact form whose Reeb flow is \emph{periodic}. For $A_k$-type, we in particular take the following contact form.
$$
\alpha : = \left.\frac{i}{4}\left( (k+1)(z_0 d\overline z_0 - \overline z_0 d z_0)  + \sum_{j=1}^n 2 (z_jd\overline z_j - \overline z_j d z_j) \right)  \right|_{\Sigma_{k}}.
$$  
The contact structure $\ker \alpha$ on $\Sigma_{k}$ is contactomorphic to the contact structure $\ker \lda|_{\Sigma_{k}}$, see \cite[Section 2]{KvK} for more details.

Its Reeb flow is given by
\begin{equation}\label{eq: Reeb A_k}
Fl_t^{R_{\alpha}}(z) = (e^{it/(k+1)}z_0, e^{it/2}z_1, \dots, e^{it/2}z_n).
\end{equation}
This is indeed periodic since $Fl_T^{R_{\alpha}}  =  \id$ for $T =  2 \pi \cdot \lcm(k+1, 2)$. We therefore have a well-defined fibered twist $\tau : \widehat W_k \rightarrow \widehat W_k$ following the recipe in Section \ref{Fibered twist}.

\subsection{Real Lagrangians}

We take a specific Lagrangian in Milnor fibers as a real Lagrangian. Consider the following map 
\begin{equation}\label{eq: conjmap}
\rho: W_{k} \rightarrow W_{k}, \quad z \mapsto (\overline z_0, -\overline z_1, \cdots, -\overline z_{n-1}, - \overline z_n),
\end{equation}
and denote its fixed point set by $ \Fix(\rho)$. Note that $\rho$ is an exact anti-symplectic involution so that the fixed point set $\Fix(\rho)$ forms (possibly disconnected) Lagrangian in $W_{k}$ (unless it is empty). It also follows that the intersection $\Fix(\rho) \cap \Sigma_{k}$ forms a Legendrian in $\Sigma_{k}$.

\begin{Rem}
One reason why we put the artificial minus signs in (\ref{eq: conjmap}) is chosen so that $\Fix(\rho)$ is not empty. For example, for $k=2$, the fixed point set of the conjugate map $z \mapsto \overline z$ restricted to $\widehat W_{2}$ is empty. 
\end{Rem}

The anti-symplectic involution (\ref{eq: conjmap}) restricts to an anti-contact involution on the boundary $\Sigma_{k}$. Its fixed point set is given by the intersection $\Fix(\rho) \cap \Sigma_{k}$. In other words,
$$
\Fix(\rho) \cap \Sigma_{k} = \{z \in \Sigma_{k} \;|\; y_0 =x_1= \cdots =x_n = 0\}
$$
where $z_j = x_j + i y_j$. We investigate topology of the real Lagrangians as follows.

%even case와 odd case가 바뀌었음%
%k가 odd라서 두개의 disjoint union으로 이루어졌을 경우에는 그 중 하나만 택해서 그것을 우리의 L이라고 해야함%
%예시로 Accordingly 이후의 statement를 아래에 써봤음%
\begin{Lem}\label{lem: topofLag}
The real Lagrangian $\Fix(\rho)$ is diffeomorphic to the ball $B^n$ if $k$ is even and is diffeomorphic to the disjoint union of the balls $B^n \sqcup B^n$ if $k$ is odd. 
\end{Lem}

\begin{proof}
Using the coordinate $z_j=x_j+iy_j$, a point $z=(z_0, \cdots, z_n) \in W_{k}$ is in $L$ if and only if $y_0=x_1= \cdots =x_n=0$. The Lagrangian $L$ is given by
$$L=\{(x_0, iy_1, \cdots, iy_n) \in \mathbb{C}^{n+1} | \quad x_0^{k+1}-|y|^2=\epsilon \cdot \zeta(x_0^2+|y|^2),\quad x_0^2+|y|^2 \le 1\}$$
where $y=(y_1, \cdots, y_n)$. Moreover, the Legendrian submanifold $\mathcal{L}$ can be written as
\begin{equation}
\label{eq: realLeg}\mathcal{L}=\{(x_0, iy_1, \cdots, iy_n) \in \mathbb{C}^{n+1} | \quad x_0^{k+1}-|y|^2=0,\quad x_0^2+|y|^2 =1\}	
\end{equation}
and thus $\mathcal{L}$ is diffeomorphic to $S^{n-1}$ if $k$ is even and diffeomorphic to $S^{n-1} \sqcup S^{n-1}$ if $k$ is odd; in the latter case, each copy of $S^{n-1}$ corresponds to positive $x_0$ and negative $x_0$ respectively. 

We consider $h(x_0, y)=x_0^2+|y|^2$ as a Morse function on the set 
$$
\{(x_0, y) \in \R^{n+1} \;|\; d(x_0, y):=x_0^{k+1}-|y|^2-\epsilon \cdot \zeta(x_0^2+|y|^2)=0 \}.
$$ 
Observe that $L$ is then the 1-sublevel set $\{h \le 1\}$. The critical point of $h$ on the hypersurface is given by
\begin{align*}
(x_0, y) \in \crit(h) &\iff \nabla h(x_0, y) \textrm{ and } \nabla d(x_0, y) \textrm{ are linearly dependent.}\\
&\iff (x_0, y) \textrm{ and } ((k+1)x_0^k, -2y) \textrm{ are linearly dependent}.
\end{align*}

First, we consider the critical point when $k$ is even. If $y \ne 0$, we have $(k+1)x_0^k=-2x_0$, so $x_0=0$ or $({-2 \over k+1})^{1 \over k-1}$. Both cases are impossible on the hypersurface $\{d(x_0, y)=0\}$ since $(x_0, y)$ satisfies $x_0^{k+1}=|y|^2+\epsilon \cdot \zeta(x_0^2+|y|^2)>0$. If $y=0$, then $x_0^{k+1}=\epsilon \cdot \zeta(x_0^2)$ and this equation has the unique solution, say $x_0^*$. Therefore we have $\crit(h)=\{(x_0^*, 0)\}$. This implies that $L$ is diffeomorphic to a ball $B^n$ if $k$ is even.

Second, let us see when $k$ is odd. It is easy to see that $L$ consists of two connected component $L \cap \{x_0>0\}$ and $L \cap \{ x_0 < 0\}$. We can show that each connected component is diffeomorphic to $B^n$ following the above argument.
\end{proof}
From now on, we set a Lagrangian $L$ in $W_k$ to be a connected component of $\Fix(\rho)$, so $L$ is diffeomorphic to the ball $B^n$. We denote its boundary by $\mathcal{L}$, which is diffeomorphic to the sphere $S^{n-1}$.

\subsection{Uniform lower bounds} 
We want to prove the following uniform lower bounds.

%\pi_0의 같은 class에 있다는 것 즉 \phi가 \tau^l  과 같은 connected component of Symp^c 에 있다는 게 가정%
\begin{Thm}\label{main thm for a_k}
Let $\tau : \widehat W_k \rightarrow \widehat W_k$ be the fibered twist corresponding to a $A_k$-type singularity. If a compactly supported symplectomorphism $\phi$ satisfies $[\phi] =  [\tau^l] \in \pi_0(\symp^c(\widehat W_k))$ for some $l \neq 0$, then $s_n(\phi) \geq 1$.
\end{Thm}

The strategy is to fit our situation into Theorem \ref{main thm} with the real Lagrangians, and we will show a linear growth of wrapped Floer homology groups using the Morse-Bott spectral sequence. It will turn out that we can compute the wrapped Floer homology $\WFH_*(W, L )$ completely, and the action filtration for the spectral sequence gives a linear growth as we want.

\begin{Rem}\
\begin{itemize}
\item It is worth pointing out that of power of the twist $\tau$ is smoothly isotopic to the identity if $n$ is even. This can be shown by considering the variation operator of $\tau$ as \cite[Proposition 2.23]{FS05} and references therein. Therefore, in such cases, we are looking at a ``symplectic phenomenon'' via the slow volume growth.

\item For $k=1$, the Milnor fiber $W_{1}$ is symplectomorphic to the disk cotangent bundle over $S^n$. In this case, our Lagrangian $L$ is exactly a cotangent fiber. In general, $W_k$ can be seen as a linear plumbing of $k-1$ cotangent bundles over $S^n$. In this sense, Theorem \ref{main thm for a_k} partially extends the result in \cite{FS05}. 
\end{itemize}
\end{Rem}

\subsubsection{Morse-Bott type Reeb chords and indices}
For $A_k$-type singularities, considering the Reeb flow (\ref{eq: Reeb A_k}) and the definition of $\mathcal{L}$ in the equation (\ref{eq: realLeg}), the minimal period of Reeb chords which start and end in $\mathcal{L}$ is given by
$$
T_k  = \begin{cases}  (k+1) \pi & \text{$k$ is even,} \\   2(k+1) \pi & \text{$k$ is odd.} \end{cases}
$$ 
Note that the period $T_k$ is realized as the half Reeb chords of the periodic Reeb orbits with the minimal common period for $k$ is even, and the periodic Reeb orbit itself for $k$ is odd. By Corollary \ref{cor: mbtype}, Reeb chords are of Morse-Bott type, and Morse-Bott components for $A_k$-type singularities are precisely given by $\mathcal{L} : = \mathcal{L}_{T_k}$ and its iterates, denoted by $N \cdot \mathcal{L}$. 

In general, the Maslov indices of periodic Reeb orbits with the minimal common period for polynomials of the form $z_0^{a_0} + \cdots + z_n^{a_n}$ is given by the following formula 
$$
\mu(\gamma) = 2 (\lcm_{j}a_j)\left(\sum_j \frac{1}{a_j}-1 \right),
$$
see \cite[Proposition 5.9]{KvK}. By Proposition \ref{lem: indrel}, the Maslov index of the corresponding half Reeb chord is then
\begin{equation}\label{eq: indform}
\mu(x) = (\lcm_{j}a_j)\left(\sum_j \frac{1}{a_j}-1 \right).
\end{equation}
Therefore we find the following.
\begin{Lem} \label{lem: indcpt}
The Maslov indices of components are given by
$$
\mu(N \cdot \mathcal{L}) = N \cdot \{2 + (n-2)(k+1)\}. 
$$
\end{Lem}

\begin{Ex}
Consider the $A_k$-singularity with $k=2$ and $n=3$. Then the situation is as follows.
\begin{itemize}
\item Morse-Bott components are given by $\mathcal{L}$ of period $6 \pi$ and its iterates $N \cdot \mathcal{L}$ of period $6N\pi$;
\item Topologically, $N \cdot \mathcal{L} \cong S^2$ for all $N \geq 1$, by Lemma \ref{lem: topofLag};
\item $\mu(N \cdot \mathcal{L}) = 5N.$
\end{itemize}
\end{Ex}

\begin{Rem}
For $A_k$-type singularities with $k \geq 2$ and $n \geq 3$, it is well known that $\pi_1(\Sigma) =0$ and $H^2(\Sigma) = 0$. Since $\mathcal{L} \cong S^{n-1}$, it also follows that $\pi_1(\Sigma, \mathcal{L}) = 0$, and the first Chern class $c_1(\xi)$ vanishes. So topological conditions for well-definedness of Maslov indices are all satisfied, see Remark \ref{rem: topcondmasReeb}. It is also known that $\pi_1(W_k,L)=0$ and $c_1(W_k)$ vanishes, so Theorem \ref{thm: ss} applies.
\end{Rem}

\subsubsection{Computations of the $E^1$-page}

Now, we have all necessary ingredients for computing the $E^1$-page of a spectral sequence in \eqref{eq: MBss}; indices and homology groups of Morse-Bott components. It turns out that the spectral sequences terminate from the $E^1$-page, that is, $d^r \equiv 0$ for $r \geq 1$:

\begin{Lem}\label{lem: ssterm}
Let $n \geq 3$ as we assumed. The spectral sequences (\ref{eq: MBss}) terminate from the $E^1$-page.
\end{Lem}

\begin{proof}
Denote the $q$-coordinate of the top generator of $H(\mathcal{L})$ by $q_{top}$ and the $q$-coordinate of the bottom generator of $H(2 \cdot \mathcal{L})$ by $q_{bot}$. It is enough to show that
$$
q_{top} < q_{bot}.
$$
Note that we can explicitly compute both of $q_{top}$ and $q_{bot}$, using Lemma \ref{lem: indcpt}. One finds
$$
q_{top} = (n-2)(k+1) + 1, \quad q_{bot} = 2(n-2)(k+1)-n+3.
$$
Therefore, under our assumption that $n \geq 3$, we always have $q_{top} < q_{bot}$.
\end{proof}

\begin{figure}
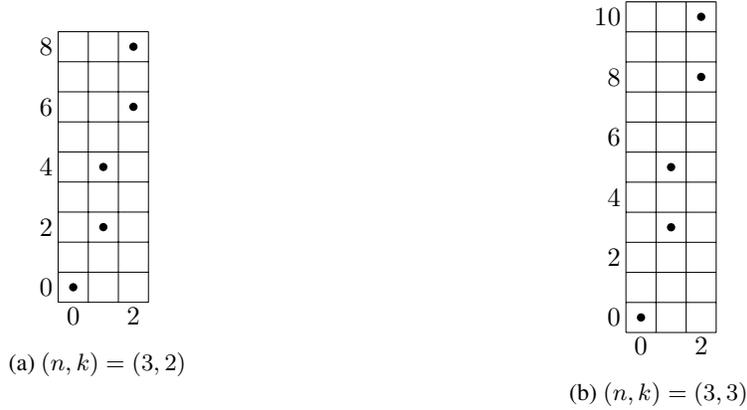

\centering

\begin{subfigure}{.5\textwidth}
  \centering
  \begin{sseq}{0...2}
{0...8}
\ssmoveto 0 0
\ssdropbull

\ssmoveto 1 2
\ssdropbull
\ssmove 0 2
\ssdropbull

\ssmoveto 2 6
\ssdropbull
\ssmove 0 2
\ssdropbull

\end{sseq}
  
    \caption{$(n,k)=(3,2)$}
  \label{fig:n=3}
\end{subfigure}%
\begin{subfigure}{.5\textwidth}
  \centering
 
 \begin{sseq}{0...2}
{0...10}
\ssmoveto 0 0
\ssdropbull

\ssmoveto 1 3
\ssdropbull
\ssmove 0 2
\ssdropbull

\ssmoveto 2 8
\ssdropbull
\ssmove 0 2
\ssdropbull

\end{sseq}

  \caption{$(n,k)=(3,3)$}
  \label{fig:n=4}
\end{subfigure}
\caption{$E^1$-pages of Morse-Bott spectral sequence}
\label{fig: ss_3222}
\end{figure}

\begin{Ex}
In Figure \ref{fig: ss_3222}, the left is the $E^1$-page for the case when $(n, k) = (3, 2)$. Each dot denotes a generator. The dot on the column $p=0$ comes from the relative homology $H_*(L, \mathcal{L})$ and the column $p=1$ comes from $H_*(L)$ and the column $p=2$ comes from $H_*(2 \cdot L)$, and so on. Observe that the differential $d^1$ of $E^1$-page vanishes identically just by the index reason; the spectral sequence terminates from the $E^1$-page. Therefore we get $\WFH_*(L;W_{2} )$ explicitly as follows.
$$
\WFH_*(L;W_{2}) = \begin{cases} \Z_2 & * = 0, 5l-2, 5l; l\in \N \\ 0 & \text{otherwise.} \end{cases}
$$
\end{Ex}

As in the above example, we get explicit formula of wrapped Floer homology groups of $A_k$-varieties with the real Lagrangians $\WFH_*(W_{k}, L)$: 

\begin{Cor}\label{cor: WFHcomp}
Let $k \geq 1$ and $n \geq 3$. The wrapped Floer homology group of $L$ in $W_k$ is given by
$$
\WFH_*(L;W_{k}) = \begin{cases} \Z_2 & *=0, \{(n-2)(k+1) +2\}l -n +1, \{(n-2)(k+1) +2\}l; l\in \N ; \\ 0 & \text{otherwise}.  \end{cases}
$$

\end{Cor}

Combining the action filtration on the $E^1$-page and Corollary \ref{cor: WFHcomp}, we conclude the following.

\begin{Cor}
For $n \geq 3$ and $k \geq 1$, the wrapped Floer homology groups $\WFH_*(L;W_{k})$ admits a linear growth. 
\end{Cor}

This completes a proof of the uniform lower bounds in Theorem \ref{main thm for a_k}. The condition $H_c^1(W;\R) = 0$ holds since the Milnor fibers are $(n-1)$-connected in general.

In addition, note that the wrapped Floer homology group $\WFH(L;W_k)$ are infinite-dimensional. By applying Theorem \ref{ThmB}, we have the following corollary, which was already proven by Seidel in \cite{Sei}

\begin{Cor}
The component of fibered twists $[\tau]$ in the group $\pi_0(\symp^c(\widehat W_k))$ has infinite order for $k \geq 1$ and $n \geq 3$.
\end{Cor}

\begin{Rem}
Computations here using the Morse-Bott spectral sequence apply to a broader class of weighted homogeneous polynomials. As a simple example, one takes a homogeneous polynomial $z_0^k + \cdots +z_n^k$, with the anti-symplectic involution given by the standard complex conjugate on the Milnor fiber. Then a connected component of the fixed point set is still a Lagrangian ball for $k$ odd, and we can compute the $E^1$-page of the spectral sequence. 

In particular, for $k = n-1$, where this is exactly the case when the index (\ref{eq: indform}) vanishes, we directly see from the $E^1$-page that the wrapped Floer homology is infinite dimensional. Therefore the corresponding component $[\tau]$ of the fibered twists has infinite order by Theorem \ref{ThmB}. 
\end{Rem}

\section{Complements of a symplectic hypersurface in a real symplectic manifold} \label{sec: chyper}
\subsection{Construction}
A triple $(M,\ow,\rho)$ is called a \emph{real} symplectic manifold if $(M,\ow)$ is a symplectic manifold and $\rho\in \Diff(M)$ is an anti-symplectic involution, i.e., $\rho^2=\id$ and $\rho^*\ow=-\ow$. Let $(M,\ow,\rho)$ be a closed, real, symplectic manifold with $[\omega]\in H^2(M;\Z)$. A symplectic submanifold $Q\subset M$ of codimension 2 is called a {\it symplectic hypersurface of degree} $k\in \N$ in $(M,\ow)$ if $[Q]\in H_{2n-2}(M;\mathbb{Z})$ is Poincar\'{e} dual to $k[\omega]$. A symplectic hypersurface $Q\subset M$ is called {\it $\rho$-invariant} if it satisfies $\rho(Q)\subset Q$. We shall prove that the complement of a $\rho$-invariant symplectic hypersurface in a real symplectic manifold is a real Liouville domain whose boundary has a periodic Reeb flow. See Proposition \ref{complement} for the precise statement.

	\begin{Rem}
The degree of a symplectic hypersurface depends on the choice of a symplectic form. Since we will consider a primitive cohomology class of a symplectic form (we can always assume this by rescaling the symplectic form), it does not cause any  ambiguity.
\end{Rem}
Let $Q$ be a $\rho$-invariant symplectic hypersurface in $M$ of degree $k\in \N$. Note that a symplectic hypersurface $(Q,k\ow|_Q,\rho|_Q)$ is an integral, real,  symplectic manifold. We construct a standard symplectic disk bundle over $(Q,k\ow|_Q)$ which admits an anti-symplectic involution induced by the push-forward of the involution $\rho|_Q$. We refer to \cite[Section 2.1]{Bir1} for details on the standard symplectic disk bundles.

Consider the symplectic normal bundle of $Q$ defined by
$$NQ:=\{v\in TM|_Q \ |\ \omega(v,w)=0 \text{ for all $w\in TQ$} \},$$
which is a symplectic vector bundle of rank 2 over $Q$. Choosing any complex structure on $NQ$, the symplectic normal bundle $NQ$ can be seen as a complex line bundle over $Q$ with $c_1(NQ)=k[\omega|_Q]$. For a fixed hermitian metric $|\cdot |$ on $NQ$, choose a connection $\nabla$ on $NQ$ with curvature $R^\nabla=2\pi ik\omega|_Q$. Denote $H^\nabla$ the horizontal distribution and $\alpha^\nabla$ the global angular 1-form on $NQ\setminus 0$ associated to $\nabla$, i.e.,
\begin{eqnarray*}
	\alpha^\nabla_u(u)=0,\quad \alpha^\nabla_u(iu)=\frac{1}{2\pi},\quad \alpha^\nabla|_{H^\nabla}=0,
\end{eqnarray*}
where $r$ is a radial coordinate on fibers induced by the metric. Let $\pi:NQ\to Q$ be the projection map. In this convention, we have $d\alpha^\nabla=-k\pi^*\omega|_Q$. Fix $0<\epsilon<1$. We call
$$N^\epsilon Q:=\{v\in NQ\ | \ |v|<\epsilon\}$$ a {\it standard symplectic disk bundle over $(Q,k\omega|_Q)$}. The total space $N^\epsilon Q$ carries a  symplectic form 
$$\Omega:=\pi^*(k\omega|_Q)+d(r^2\alpha^\nabla).$$
Note that there is a Liouville form on $N^\epsilon Q\setminus 0$, namely, 
$\lambda_0=(r^2-1)\alpha^\nabla$. The Liouville vector field $X_{\lambda_0}$, which is the solution to $\Omega(X_{\lambda_0},\cdot)=\lambda_0$, is given by
$\displaystyle X_{\lambda_0} = \frac{r^2-1}{2r}\frac{\partial}{\partial r}.$
Since $X_{\lambda_0}$ points inward to the zero section and is transverse to a circle bundle $P_\delta:=\{|v|=\delta \}$ for any $0<\delta<\epsilon$, the Liouville form $\lambda_0$ restricts to a contact form ${\lambda_0}|_{P_\delta}$ on $P_\delta$.  Note that a global angular 1-form $\alpha^\nabla$ also restricts to a contact form on $P_\delta$. In this case, a contact manifold $(P_\delta,\xi=\ker \alpha^\nabla)$ is called a {\it prequantization bundle }(or Boothby-Wang bundle) over $(Q,k\omega|_Q)$. A prequantization bundle has a periodic Reeb flow with the same period for all simple periodic orbits, and simple periodic orbits are given by fiber circles. See \cite{BW}, \cite[Section 7.2]{Gei} for details on prequantization bundles.

Since a symplectic hypersurface $Q$ is $\rho$-invariant, the push-forward $\rho_*:TM\to TM$ of $\rho$ descends to an involution $\mathcal{R}:=\rho_*: NQ \to NQ$ that preserves  fibers, i.e., for $x\in Q$, we have $\mathcal{R}(N_xQ)=N_{\rho(x)}Q$. Assume that the hermitian metric on $NQ$ is chosen to be $\rho$-invariant, so $\mathcal{R}$ restricts to an involution on $N^\epsilon Q$.

\begin{Lem}\label{comp_anti}
The involution $\mathcal{R}:N^\epsilon Q\to N^\epsilon Q$ is anti-symplectic with respect to the symplectic form $\Omega$, i.e., $\mathcal{R}^*\Omega=-\Omega$.
\end{Lem}
\begin{proof}
Note that $\Omega=k(1-r^2)\pi^*\ow|_Q+2rdr\wedge \alpha^\nabla$. Since every linear anti-symplectic involution on a 2-dimensional symplectic vector space is a reflection, it follows that $\mathcal{R}^*(rdr\wedge\alpha^\nabla)=-rdr\wedge \alpha^\nabla$. Using the identity $\pi\circ \mathcal{R}=\rho|_Q\circ \pi$, we see that $\mathcal{R}^*(\pi^*\omega|_Q)=\pi^*(\rho|_Q^*\omega|_Q)=-\pi^*\omega|_Q$. This proves that $\mathcal{R}^*\Omega=-\Omega$. 
\end{proof} 
The following lemma tells us that a standard symplectic disk bundle equipped with the anti-symplectic involution $\mathcal{R}$ is the standard model to a tubular neighborhood of a $\rho$-invariant symplectic hypersurface. 
\begin{Lem}\label{neighborhood of hyp}
	There is a $\rho$-invariant tubular neighborhood $\nu_M(Q)$ of $Q$ in $M$ and a symplectomorphism $\Phi:(\nu_M(Q),\omega)\to (N^{\epsilon} Q,\Omega)$ (with possibly smaller $\epsilon$) such that $\mathcal{R}\circ \Phi=\Phi\circ \rho$.
\end{Lem}
A proof is an application of the Moser argument. We refer to \cite[Theorem 2]{Mey} for a proof. In there, it is proven that a neighborhood of a real Lagrangian $L$ in a real symplectic manifold is identified with a neighborhood of the zero section in  the cotangent bundle $T^*L$ of the real Lagrangian with an anti-symplectic involution defined by $v\mapsto -v$ for $v\in T^*Q$. We now prove the following proposition. See also \cite[Lemma 3.1]{FvK}.
\begin{prop}\label{complement}
		Let $(M,\omega,\rho)$ be a closed, real symplectic manifold with $[\ow]\in H^2(M;\Z)$ and let $Q\subset M$ be a $\rho$-invariant symplectic hypersurface in $M$ of degree $k$.  Then, the complement of $Q$ in $M$,	$$(W:=M\setminus \nu_M(Q),\ \omega|_W=d\lambda,\ \rho|_W),$$ is a real Liouville domain whose contact boundary $(\Sigma, \xi:=\ker\lambda|_\Sigma)$ is a prequantization bundle over $(Q,k\omega|_Q)$, where $\nu_M(Q)$ is a $\rho$-invariant tubular neighborhood of $Q$ in $M$.
	\end{prop}

	\begin{proof}
Let $\Phi:\nu_M(Q)\to N^\epsilon Q$ be a symplectomorphism as in Lemma \ref{neighborhood of hyp}. By abusing of the notation, let $\nu_M(Q)=\Phi^{-1}(N^\delta Q)$ be a smaller $\rho$-invariant tubular neighborhood where $\delta<\epsilon$ is fixed. A compact manifold $W=M\setminus \nu_M(Q)$ with boundary $\Sigma=\Phi^{-1}(P_\delta)$ admits a Liouville form as follows. Let $\pi:N^\delta Q\to Q$ be the projection. By \cite[Lemma 2.1]{Op}, choose a closed 1-form $\eta$ on $Q$ such that a Liouville form $\lambda_\eta:=\Phi^*(\lambda_0+\pi^*\eta)$ defined on $\nu(Q)\setminus Q$ can be extended to a Liouville form $\tilde{\lambda}_\eta$ on $M\setminus Q$. One can check that $\tilde{\lambda}_\eta$ restricts to a contact form on the boundary $\Sigma$, and a contact manifold $(\Sigma,\ker \tilde{\lambda}_\eta|_\Sigma)$ is contactomorphic to a prequantization bundle  by the Gray's stability theorem. Since the tubular neighborhood $\nu_M(Q)$ is $\rho$-invariant, the involution $\rho\in \Diff(M)$ restricts to an anti-symplectic involution on $W$. Finally, the Liouville form $\lambda:=\frac{1}{2}(\tilde{\lambda}_\eta-\rho^*\tilde{\lambda}_\eta)$ on $W$ satisfies $\rho^*\lambda=-\lambda$ as desired.
	\end{proof}
	Note that an anti-symplectic involution $\rho\in \Diff(M)$ restricts to (exact) anti-symplectic involutions on $(W,\omega)$ and $(Q,\omega|_Q)$  and to an anti-contact involution on $(\Sigma,\alpha=\lambda|_\Sigma)$. We will write $\rho$ as an involution in various situation when it does not cause confusion. Furthermore, we always assume that the fixed points sets, $\Fix(\rho|_W)$, $\Fix(\rho|_\Sigma)$ and $\Fix(\rho|_Q)$, are all non-empty.		
\subsection{Setup}\label{comp_setup}
From now on, we assume the following setup.
\begin{itemize}
	\item Let $(M^{2n},\omega,\rho)$ be a closed, real, symplectic manifold such that $[\omega]\in H^2(M;\Z)$ is primitive and $n\ge 3$. 
	\item There exists $c\in \Z$ such that $c_1(M)=c[\omega]$, i.e., $(M,\omega)$ is monotone.
	\item Let $Q\subset M$ be a $\rho$-invariant symplectic hypersurface of degree $k$ which is simply connected.
	\item A symplectic hypersurface $Q$ is of Donaldson type, see \cite{Don}, \cite[Section 2.3]{Bir2} and \cite[Section 10]{CDvK}.
\end{itemize}
By Proposition \ref{complement}, the complement $(W=M\setminus \nu_M(Q),\omega|_W=d\lambda,\rho|_W)$ is a real Liouville domain whose boundary $(\Sigma,\xi)$ is a prequantization bundle equipped with the anti-contact involution. Since $Q$ is a symplectic Donaldson hypersurface, $(W,\lambda)$ is a \emph{Weinstein domain} by \cite[Proposition 11]{Gir}. In this paper, we will not define the notion of Weinstein domains, but we only use the following fact: Weinstein domains have the homotopy type of a CW-complex whose indices are at most $n=\frac{1}{2}\dim W$, see \cite[Lemma 11.13]{CE}. In Section \ref{comp_exam}, we provide some explicit examples satisfying the assumptions in the setup.
\begin{Rem}\label{remark_assumption} \
\begin{itemize}
	\item If $(M, \ow,J)$ is a compact K\"{a}hler manifold with $[\ow]\in H^2(M;\Z)$, then any smooth complex hypersurface in $M$ whose homology class is Poincar\'{e} dual to $k[\omega]$ for $k\in \N$, is a symplectic Donaldson hypersurface in a sense that the complement $W=M\setminus \nu_M(Q)$ is a Weinstein domain, see \cite[Section 7]{Bir1}.
	\item Since Weinstein domains have a handle decomposition whose indices are at most $n={1 \over 2} \dim W$, it follows that $\pi_i(W,\Sigma)=0$ for $1 \le i\le n-1$ and $H^{2n-2}(W;\Z)=0$ for $n\ge 3$. Moreover, the topological assumption in Theorem \ref{main thm}, i.e., $H^1_c(W;\R)\cong H_{2n-1}(W;\R)=0$, trivially holds for $n\ge 2$.
	\item  If $\pi_1(M)=0$, then $\pi_1(Q)=0$ by an symplectic analogue of the Lefschetz hyperplane theorem \cite[Proposition 39]{Don}.
	\item It is worth noting that for any large enough degree $k$ there are symplectic Donaldson hypersurfaces in an integral symplectic manifold. This is due to Donaldson \cite{Don}.
	\end{itemize}
\end{Rem}\noindent
%\subsection{Topology of the complement and its boundary}
We give the topological properties on the complement $(W,\ow)$ and its boundary $(\Sigma,\xi)$ needed to apply the spectral sequence in Theorem \ref{thm: ss}.
\begin{Lem}\label{comp_subtop}
A symplectic hypersurface $(Q,\omega|_Q)$ is monotone with $c_1(Q)=(c-k)[\omega|_Q]$ and $[\omega|_Q]\in H^2(Q)$ is primitive.
\end{Lem}
\begin{proof}
	Since $TM|_Q\cong TQ\oplus NQ$ and $c_1(NQ)=k[\omega|_Q]$, it follows that $c_1(Q)=c_1(TM|_Q)-c_1(NQ)=(c-k)[\omega|_Q]$. We can verify that $H_2(M,Q)\cong H_2(W,\partial W)\cong H^{2n-2}(W)=0$, see Remark \ref{remark_assumption}. Here, we used the excision to obtain the first isomorphism. By a long exact sequence of a pair $(M,Q)$, it implies that the map $i_*\colon H_2(Q) \to H_2(M)$ induced by the inclusion is surjective. Since $[\omega]\in H^2(M)$ is primitive, choose $[\sigma]\in H_2(M)$ such that $\langle [\omega],[\sigma]\rangle=1$. Take $[\sigma']\in H_2(Q)$ with $i_*[\sigma']=[\sigma]$. Then, we see that $\langle [\omega|_Q],[\sigma']\rangle=1$. This shows that $[\omega|_Q]\in H^2(Q)$ is primitive.
\end{proof}

\begin{Lem}\label{comp_top} We have that
\begin{enumerate}[label=(\arabic*)]
	\item $\pi_1(\Sigma)\cong \pi_1(W)\cong \Z_k$.
	\item $c_1(W)$ vanishes on $\pi_2(W)$.
	\item $c_1(\xi)$ vanishes on $\pi_2(\Sigma)$.
\end{enumerate}
\end{Lem}
\begin{proof} The first is already known in \cite[Remark 4.17]{CDvK}. Here, we give the details. The homotopy exact sequence of a circle bundle $\Sigma\to Q$ implies that $\pi_1(\Sigma)$ is abelian, so $\pi_1(\Sigma)\cong H_1(\Sigma)$. Recall that $[\omega|_Q]$ is primitive by Lemma \ref{comp_subtop}. The relevant part of the (homological) Gysin sequence,
$$\begin{CD}
H_2(Q)@>\cap k[\omega|_Q]>>H_0(Q)\cong \mathbb{Z}@>>>H_1(\Sigma)@>>>H_1(Q)\cong 0
\end{CD}$$
tells us that $\pi_1(\Sigma)\cong H_1(\Sigma)\cong \mathbb{Z}_k$. By Remark \ref{remark_assumption}, we get $\pi_1(\Sigma)\cong \pi_1(W)$. Since $\omega|_W$ is exact, we verify that $\langle c_1(W),[\sigma] \rangle=\langle c[\omega|_W],[\sigma] \rangle=0$ for $[\sigma]\in \pi_2(W)$. Note that $TW|_\Sigma\cong \xi\oplus\langle R_\alpha, X \rangle$, where $R_\alpha$ is a Reeb vector field and $X$ is the Liouville vector field associated to $\lambda$. Hence, we obtain $c_1(TW|_\Sigma)=c_1(\xi)$ and that $c_1(\xi)$ vanishes on $\pi_2(\Sigma)$.
\end{proof}
To apply Theorem \ref{main thm}, we will consider an admissible Lagrangian ball  $L\cong B^n$ in $W$. In later examples, some connected component of a real Lagrangian in $W$ will provide such a Lagrangian, see Lemma \ref{comp_ex1top} and \ref{comp_ex2top}. The following is an immediate consequence by using appropriate relative exact sequences.
\begin{Lem}\label{lem_topofcomp}
 Let $L\subset W$ be an admissible Lagrangian with a Legendrian boundary $\mathcal{L}=\p L$ such that $\pi_1(L)=\pi_1(\mathcal{L})=0$. Then, we have $\pi_1(\Sigma,\mathcal{L})\cong \pi_1(W,L)\cong \Z_k$ and the map $i_*:\pi_1(\Sigma,\mathcal{L})\to \pi_1(W,L)$ induced by the inclusion is bijective.
\end{Lem}

\subsection{Reeb chords on a prequantization bundle}
We shall exhibit Reeb chords on the prequantization bundle $(\Sigma,\alpha)$  over $(Q,k\ow|_Q,\rho|_Q)$ obtained in Section \ref{comp_setup}. Recall that $(\Sigma,\alpha)$ carries an anti-contact involution $\rho\in \Diff(\Sigma)$ whose fixed point set $\Fix(\rho)$ consists of (possibly) several connected components. Assume that the real Lagrangian $\Fix(\rho|_Q)$ in $Q$ is path-connected. Our examples in Section \ref{comp_exam} will fulfill this condition. The Legendrian $\Fix(\rho)$ has at most two components as the natural projection $\pi:\Fix(\rho)\to \Fix(\rho|_Q)$ is a 2-fold covering. In this section, we consider the following two cases: either 
\begin{enumerate}[label=\bf{(C\arabic*)}]
	\item \label{c1} $\Fix(\rho)$ is path-connected, or
	\item \label{c2} $\Fix(\rho)$ consists of two connected components.
\end{enumerate}
In both cases, we let $\mathcal{L}$ be a Legendrian in $\Sigma$ given by a connected component of $\Fix(\rho)$.
\begin{Lem}\label{preq_morsebott}
Reeb chords of $\mathcal{L}$ are of Morse-Bott type. More precisely, we have
	\begin{enumerate}[label=(\arabic*)]
		\item $\text{Spec}(\Sigma,\alpha,\mathcal{L})=T_0\cdot \N$ for some $T_0>0$.
		\item For $T\in \text{Spec}(\Sigma,\alpha,\mathcal{L})$, a Morse-Bott component $\mathcal{L}_T$ is diffeomorphic to $\mathcal{L}$.
	\end{enumerate}	
\end{Lem}
\begin{proof}
We denote by $T_P$ the minimal period of the Reeb flow on $(\Sigma,\alpha)$. Let $z\in \mathcal{L}$ such that $Fl_T^{R_\alpha}(z)\in \mathcal{L}$ for $T>0$, i.e., $\rho(Fl_T^{R_\alpha}(z))=Fl_T^{R_\alpha}(z)$. Then we verify that
$$
z=\rho(z)=\rho\circ Fl_{-T}^{R_\alpha}\circ Fl_T^{R_\alpha}(z)= Fl_T^{R_\alpha}\circ \rho \circ Fl_T^{R_\alpha}(z)=Fl_{2T}^{R_\alpha}(z),
$$
which implies that $2T$ is a multiple of $T_p$. Note that
$$
\Spec(\Sigma,\alpha,\mathcal{L})=T_0\cdot \N,
$$
where $T_0=\begin{cases}
	\frac{1}{2}T_P & \text{for \ref{c1}} \\
	T_P &\text{for \ref{c2}}
\end{cases}$ is given by a period of a Reeb chord of the minimal period. In both cases, every Morse-Bott component $\mathcal{L}_T$ is diffeomorphic to $\mathcal{L}$ by \eqref{eq: mbc} and it follows from Corollary \ref{cor: mbtype} that Reeb chords are Morse-Bott type.
\end{proof}
\begin{Rem}
Every Reeb chord is an iterate of a Reeb chord $c_P:[0,T_0]\to \Sigma$ of the minimal period $T_0$, and $c_P$ is given by
\begin{itemize}
	\item a half Reeb chord of a simple periodic Reeb orbit starting at a point in $\mathcal{L}$ for the case \ref{c1}, and
	\item a simple periodic Reeb orbit starting at a point in $\mathcal{L}$ for the case \ref{c2}. 
\end{itemize}	
\end{Rem}
We denote by $c_P^l$ a $l$-th iterate of $c_P$ for $l\in \N$. Write $\gamma_P:[0,T_P]\to \Sigma$ for a simple periodic Reeb orbit starting at a point in $\mathcal{L}$ and its $l$-fold cover is denoted by $\gamma_P^l$.
  The following lemma characterizes contractible Reeb chords.
\begin{Lem}\label{preq_cont}
Suppose that $\pi_1(\mathcal{L})=0$. A Reeb chord $c_P^l$ is contractible if and only if $k$ divides $l$.
	\end{Lem}
	
	\begin{proof}
	Assume the case \ref{c1}. We observe that any two half Reeb chords of simple periodic Reeb orbits starting at (different) points in $\mathcal{L}$ are homologous as $\mathcal{L}$ is path-connected. Since the canonical map  $j_*:H_1(\Sigma)\to H_1(\Sigma,\mathcal{L})$ is an isomorphism given by $[\gamma_P] \mapsto 2[c_P]$ and $H_1(\Sigma)\cong H_1(\Sigma,\mathcal{L}) \cong \mathbb{Z}_k$ by Lemma \ref{lem_topofcomp}, we observe that $k$ should be odd. Indeed, if $k=2m$ is even we get the contradiction $0\ne j_*(m[\gamma_P])=2m[c_P]=k[c_P]=0$. We now see that
$$2l[c_P]=[c_P^{2l}]=[\gamma_P^l]
		=l[\gamma_P]=0 \quad \iff \quad \text{$k \mid l$}.$$
	Since $k$ is odd, the condition $2l[c_P]=0$ is equivalent to the condition $l[c_P]=[c_P^l]=0$. 
	
Now, consider the case \ref{c2}. Since $c_P^l$ is given by a $l$-fold cover of a simple periodic Reeb orbit starting at a point in $\mathcal{L}$, the result follows from Lemma \ref{comp_top}.
	\end{proof}

\subsection{Computation of the wrapped Floer homology}
We are ready to apply the Morse-Bott spectral sequence (Theorem \ref{thm: ss}) in our situation. Recall that the complement $(W^{2n},\ow,\rho)$ is the real Liouville domain satisfying the setup in Section \ref{comp_setup} and its boundary $(\Sigma,\alpha)$ is a prequantization bundle equipped with the anti-contact involution $\rho|_\Sigma\in \Diff(\Sigma)$. We remind that $k\in \N$ is the degree of a $\rho$-invariant symplectic hypersurface in a real symplectic manifold $(M,\ow,\rho)$ and $c_1(M)=c[\ow]$ for some $c\in \N$.

\begin{prop}\label{comp_prop} 
Let $n\ge 3$. We assume that one of the following holds.
\begin{enumerate}[label=(\arabic*)]
\item The real Lagrangian $\Fix(\rho)$ in $W$ is diffeomorphic to the ball $B^n$, and either $c-k>n$ or $c-k<2-n$ holds. 
\item The real Lagrangian $\Fix(\rho)$ in $W$ consists of two connected components such that at least one component is diffeomorphic to the ball $B^n$, and either $2(c-k)>n$ or $2(c-k)<2-n$ holds.
\end{enumerate}
Let $L$ be the connected component of $\Fix(\rho)$ diffeomorphic to the ball $B^n$. Then, the wrapped Floer homology group of the Lagrangian $L$ in $W$ has a linear growth, and is given by
$$
\WFH_*(L;W) = \begin{cases} \Z_2 & *=0, \mu_P\cdot l -n +1, \mu_P\cdot l; l\in \N ; \\ 0 & \text{otherwise},  \end{cases}
$$
where $\mu_P=\begin{cases}
	c-k & \text{for \ref{c1},} \\
	2(c-k) & \text{for \ref{c2}}.
\end{cases}$
\end{prop}
\begin{proof}
 Since the Legendrian $\mathcal{L}:=\partial L\cong S^{n-1}$ is a connected component of the fixed point set $\Fix(\rho|_\Sigma)$ of the anti-contact involution $\rho|_\Sigma$, Reeb chords are of Morse-Bott type by Lemma \ref{preq_morsebott}. Hence, the Morse-Bott spectral sequence (Theorem \ref{thm: ss}) applies. Note that by Lemma \ref{lem_topofcomp} the topological condition for the spectral sequence is satisfied. Now, we need to compute the Maslov index of contractible Morse-Bott components. It follows from Lemma \ref{preq_cont} that contractible Morse-Bott components are given by the iterates of $\mathcal{L}_{kT_0}$, where $T_0>0$ is the minimal period of Reeb chords. We write $N\cdot \mathcal{L}$ for the $N$-th iterate of $\mathcal{L}_{kT_0}$. We claim that the Maslov index of a contractible Morse-Bott component $N\cdot \mathcal{L}$ is given by
 $$
 \mu(N\cdot \mathcal{L})=\begin{cases}
 	N(c-k) & \text{for \ref{c1},}\\
 	2N(c-k) & \text{for \ref{c2}.}
 \end{cases}
 $$
In \ref{c1}, a Reeb chord $c_P$ of the minimal period $T_0$ is a half Reeb chord of a simple periodic Reeb orbit $\gamma_P$ starting at a point in $\mathcal{L}$. Since a $k$-fold cover of $\gamma_P$ is contractible and its Maslov index is given by $\mu(\gamma_P^k)=2(c-k)$ (see \cite{Bou}, \cite{CDvK}, \cite[Section 4.2]{Ka}), we obtain $\mu(c_P^k)=\frac{1}{2}\mu(\gamma_P^k)=n-k$ by Proposition \ref{lem: indrel}. In \ref{c2}, we know that $c_P=\gamma_P$, so we get $\mu(c_P^k)=2(c-k)$. By the concatenation property of the Maslov index, we prove the claim.

On the other hand, our assumption implies that either $\mu(1\cdot \mathcal{L})>n$ or $\mu(1\cdot \mathcal{L})<2-n$ is satisfied. If $\mu(1\cdot \mathcal{L})>n$, then the same argument in the proof of Lemma \ref{lem: ssterm} shows that the spectral sequence \eqref{eq: MBss}  terminates at the $E^1$-page. On the other hand, the condition $\mu(1\cdot \mathcal{L})<2-n$ is equivalent to the condition that the $q$-coordinate of the bottom generator of $H(1\cdot \mathcal{L})$ is strictly larger than the $q$-coordinate of the top generator of $H(2\cdot \mathcal{L})$ in the spectral sequence \eqref{eq: MBss}. This also implies that the spectral sequence terminates at the $E^1$-page, so the computation follows. Since the period of a Morse-Bott component $N\cdot \mathcal{L}$ increases linearly in the iteration number $N$, we obtain a linear growth of $\WFH_*(L;W)$.
\end{proof}
Therefore, under the assumption in Proposition \ref{comp_prop}, we can apply Theorem \ref{ThmA} and \ref{ThmB}.
\begin{Rem} \
\begin{itemize}
	\item The index condition in Proposition \ref{comp_prop} is satisfied if $c$ or $k$ is sufficiently large. We will see concrete examples corresponding to each case in the next subsection.
	\item The index $\mu_P=\mu(c_P^k)$ is the Maslov index of the smallest contractible iterate of a Reeb chord of the minimal period.
\end{itemize}
\end{Rem}

\subsection{Explicit examples}\label{comp_exam}
We introduce two examples satisfying the assumption in Section \ref{comp_setup}. Moreover, it turns out that the real Lagrangian in these examples always has a connected component which is diffeomorphic to the ball. Hence, we can apply Proposition \ref{comp_prop} to achieve the uniform lower bounds.
\subsubsection{The complement of a smooth complex hypersurface in $\C P^n$}\label{example_cp}

		Let $\mathbb{C} P^n$ be a complex projective space with the Fubini-Study symplectic form $\omega_{FS}$. We normalize the symplectic form to satisfy $\langle [\omega_{FS}],[\mathbb{C}P^1]\rangle=1$, so $[\omega_{FS}]\in H^2(\C P^n;\Z)$ is a primitive class. Note that $c_1(\C P^n)=(n+1)[\ow_{FS}]$. Define an anti-symplectic involution $$\rho([z_0:\cdots:z_n])=[\bar{z}_0:\cdots:\bar{z}_n]$$ as a complex conjugate and we identify $\Fix(\rho)$ with $\R P^n$ via $$[z_0:\cdots :z_n]\in \Fix(\rho)\longmapsto [\text{Re} z_0:\cdots:\text{Re} z_n]\in \R P^n.$$ Consider a homogeneous polynomial $$f_k(z_0,\cdots,z_n)=\begin{cases}
			z_0^k+z_1^k+\cdots +z_{n-1}^k+z_n^k & \text{if $k\in \N$ is odd}. \\
			z_0^k+z_1^k+\cdots +z_{n-1}^k-z_n^k & \text{if $k\in \N$ is even}.
		\end{cases}$$ Then, $Q_k=\{f_k=0\}\subset \C P^n$ is a smooth complex hypersurface of degree $k$. By Lefschetz hyperplane theorem, $Q_k$ is simply connected for $n\ge 3$. Since the coefficients in the polynomial $f_k$ are real, $Q_k$ is $\rho$-invariant. By Proposition \ref{complement}, the complement $W_k:=\mathbb{C} P^n \setminus \nu(Q_k)$ is a real Liouville domain whose boundary is a prequantization bundle over $(Q_k,k\omega_{FS}|_{Q_k})$. For instance, if $k=1$, then $W_1\cong B^{2n}$ is a standard symplectic ball $B^{2n}$ and $\partial W_1\cong S^{2n-1}\to \C P^{n-1}$ is the Hopf fibration.

We shall prove that the real Lagrangian $\Fix(\rho|_{W_k})$ in $W_k$ has a connected component which is diffeomorphic to the ball $B^n$ for all $k\in \N$. For the convenience, let $L_{\C P^n}:=\Fix(\rho)\cong \R P^n$ and $L_{Q_k}:=\Fix(\rho|_{Q_k})\subset Q_k$ be the real Lagrangians in $\C P^n$ and $Q_k$, respectively. In view of the decomposition $M=W\cup \nu_M(Q)$, we can write $L_{\C P^n}=\Fix(\rho|_{W_k})\cup \nu_{L_{\C P^n}}(L_{Q_k})$, where $\nu_{L_{\C P^n}}(L_{Q_k})$ is a tubular neighborhood of $L_{Q_k}$ in $L_{\C P^n}$. This decomposition is useful to determine the topology of the real Lagrangian $\Fix(\rho|_{W_k})$ in $W_k$.
		\begin{Lem}\label{comp_ex1top}\
		\begin{enumerate}[label=(\arabic*)]
			\item If $k$ is odd, then $\Fix(\rho|_{W_k})$ is diffeomorphic to the ball $B^n$.
			\item If $k$ is even, then $\Fix(\rho|_{W_k})$ consists of two connected components, and one of them is diffeomorphic to the ball $B^n$.
		\end{enumerate}

		\end{Lem}
		
		\begin{proof}
			\emph{Case 1. $k$ is odd.}
			
			Let $X_k=\{(x_0,\cdots,x_n)\in \R^{n+1}\mid x_0^k+\cdots+x_n^k=1\}$ be a smooth variety diffeomorphic to $\R^n$. We define a diffeomorphism
			$$
			\Phi:X_k\to L_{\C P^n}\setminus L_{Q_k},\quad \Phi(x_0,\cdots,x_n)\mapsto [x_0:\cdots:x_n],
			$$
			with its inverse map $\displaystyle \Phi^{-1}([x_0:\cdots:x_n])=\frac{(x_0,\cdots,x_n)}{(\sum_jx_j^k)^{1/k}}$. Hence, $\Fix(\rho|_{W_k})=L_{\C P^n}\setminus \nu(L_{Q_k})$ is diffeomorphic to the ball.\\\\
			\emph{Case 2. $k$ is even.}
			
In this case, $L_{\C P^n}\setminus L_{Q_k}$ has two connected components, namely,
\begin{eqnarray*}
	L^0 &:=&\{ [x_0:\cdots: x_n] \ |\ x_0^k+\cdots+x_{n-1}^k<x_n^k,\  x_n\ne 0\} \\
	L^1 &:=& \{ [x_0:\cdots: x_n] \ |\ x_0^k+\cdots+x_{n-1}^k>x_n^k\}
\end{eqnarray*}
and we check that $L^0$ is diffeomorphic to $A:=\{(y_0,\cdots,y_{n-1})\in \R^n\ |\ y_0^k+\cdots +y_{n-1}^k<1\}\cong \R^n$ via the diffeomorphism
$$
\Phi:A\to L^0,\quad \Phi(y_0,\cdots,y_{n-1})=[y_0:\cdots:y_{n-1}:1]
$$ with the inverse map $\displaystyle \Phi^{-1}([x_0:\cdots:x_n])=\left(\frac{x_0}{x_n},\cdots,\frac{x_{n-1}}{x_n}\right)$. 
The connected component of $\Fix(\rho|_{W_k})$ contained in $L^0$ is diffeomorphic to the ball $B^n$.
		\end{proof}

We thus take an admissible Lagrangian $L$ in $W_k$ as a connected component of the real Lagrangian $\Fix(\rho|_{W_k})$ that is diffeomorphic to the ball $B^n$. As a result, Proposition \ref{comp_prop} is suitably applied when the degree $k\in \N$ is sufficiently large as follows.
\begin{Cor}\label{comp_exam1cor}
Let $n\ge 3$. Suppose that either
\begin{itemize}
	\item $k$ is odd and $k>2n-1$, or
	\item $k$ is even and $k>\left\lfloor \frac{3}{2}n \right\rfloor$.
\end{itemize}
Then, the wrapped Floer homology group of a Lagrangian $L$ in $W_k$ has a linear growth, and is given by
$$
\WFH_*(L;W_k) = \begin{cases} \Z_2 & *=0, \mu_P\cdot l -n +1, \mu_P\cdot l; l\in \N ; \\ 0 & \text{otherwise},  \end{cases}
$$
where $\mu_P=\begin{cases}
	n+1-k & \text{if $k$ is odd},\\
	2(n+1-k) & \text{if $k$ is even}.
\end{cases}$
\end{Cor}
As we desired, we obtain the following results by Theorem \ref{ThmA} and \ref{ThmB}.
\begin{Thm}\label{main thm for comp1}
Let $\tau : \widehat W_k \rightarrow \widehat W_k$ be the twist. Under the assumption in Corollary \ref{comp_exam1cor}, we have that
\begin{itemize}
	\item If a compactly supported symplectomorphism $\phi$ satisfies $[\phi] =  [\tau^l] \in \pi_0(\symp^c(\widehat W_k))$ for some $l \neq 0$, then $s_n(\phi) \geq 1$.
	\item a class $[\tau]$ has an infinite order in $\pi_0(\symp^c(\widehat{W}_k))$.
\end{itemize}
\end{Thm}

\begin{Rem}
In \cite{Sei}, it is shown that the component of $[\tau]$ has infinite order in $\pi_0(\symp^c(\widehat{W}_k))$ for $k\ne 1$. If $k=1$, then $W_1\cong (B^{2n},\omega_{\text{std}})$ and $\tau$ is compactly supported symplectically isotopic to the identity \cite[Examples 4.1 (a)]{Sei}.
\end{Rem}

\subsubsection{The complement of a hyperplane section in a complex hypersurface}\label{sec: comp_ex2}
Let $M_d\subset (\mathbb{C}P^{n+1},\ow_{FS})$ be a complex hypersurface of degree $d\in \N$ (given by the same polynomial in Section \ref{example_cp}) and let $H=\{z_0=0\}\cap M_d$ be a hyperplane section (so it is a hypersurface of degree 1). Note that $[\ow|_{M_d}]$ is primitive and $c_1(M_d)=(n+1-d)[\ow_{FS}|_{M_d}]$ by Lemma \ref{comp_subtop}. Again, the complement $W_d^{2n}=M_d\setminus \nu_{M_d}(H)$ is a real Liouville domain by Proposition \ref{complement}. By Lefschetz hyperplane theorem, $M_d$ and $H$ are simply connected for $n\ge 3$. 
%\begin{Rem}
%We will not use this fact, but one can regard this complement $W_d$ as the Milnor fiber of a homogenous polynomial of degree $d\in \N$. 	
%\end{Rem}

In the following, we determine the topology of the real Lagrangian in the complement $W_d$. The result is similar to the situation in the Milnor fibers of $A_k$-type singularities, see Lemma \ref{lem: topofLag}. 
\begin{Lem}\label{comp_ex2top}
		The real Lagrangian $\Fix(\rho|_{W_d})$ in $W_d$ is diffeomorphic to the ball $B^n$ if $d$ is odd and is diffeomorphic to the disjoint union of the balls $B^n \sqcup B^n$ if $d$ is even. 
\end{Lem}
\begin{proof}
From the decomposition $\Fix(\rho|_{M_d})=\Fix(\rho|_{W_d})\cup \nu(\Fix(\rho|_H))$, it suffices to show that $\Fix(\rho|_{M_d})\setminus \Fix(\rho|_H)$ is diffeomorphic to $\R^n$ or $\R^n\sqcup\R^n$ for $k$ odd or $k$ even, respectively. \\\\
\emph{Case 1. $d$ is odd.}

Let $X_d$ be a smooth variety diffeomorphic to $\R^n$ given in Lemma \ref{comp_ex1top}. We have a diffeomorphism
$$
		\Phi : X_d \longrightarrow  \Fix(\rho|_{M_d})\setminus \Fix(\rho|_H),\quad \Phi(x_1,\cdots,x_{n+1}) \longmapsto [-1:x_1:\cdots:x_{n+1}]
$$
	with the inverse map $\displaystyle \Phi^{-1}([x_0:\cdots:x_{n+1}])=\frac{(x_1,\cdots,x_{n+1})}{-x_0}$. Hence, $\Fix(\rho|_{W_d})$ is diffeomorphic to the ball $B^n$.\\\\
\emph{Case 2. $d$ is even.}

Let $A=\{(x_0,\cdots,x_{n})\ |\ x_0^k+\cdots+x_{n}^k=1,\ x_1\ne 0\}$ be the space in $\R^{n+1}$ diffeomorphic to $\R^n\sqcup \R^n$. One has a diffeomorphism defined by
$$
\Psi:A \longrightarrow \Fix(\rho|_{M_d})\setminus \Fix(\rho|_H),\quad \Psi(x_0,\cdots,x_{n}) \longmapsto [x_0:\cdots:x_n:1]
$$
with the inverse map $\displaystyle \Psi^{-1}([x_0:\cdots:x_{n+1}])=\frac{(x_0,\cdots,x_{n})}{x_{n+1}}$. Hence, $\Fix(\rho|_{W_d})$ is the disjoint union of the balls $B^n\sqcup B^n$.
\end{proof}

We set an admissible Lagrangian $L\subset W_d$ to be a connected component of the real Lagrangian $\Fix(\rho|_{W_d})$. Then, $L$ is diffeomorphic to the ball $B^n$ as the above lemma. For sufficiently large degree $d\in \N$, Proposition \ref{comp_prop} gives the following corollary immediately.
\begin{Cor}\label{comp_exam2cor}
Let $n\ge 3$. Suppose that either
\begin{itemize}
	\item $d$ is odd and $d>2n-2$, or
	\item $d$ is even and $d>\lfloor \frac{3}{2}n \rfloor-1$.
\end{itemize}
Then, the wrapped Floer homology group of a Lagrangian $L$ in $W_d$ has a linear growth, and is given by
$$
\WFH_*(L;W_d) = \begin{cases} \Z_2 & *=0, \mu_P\cdot l -n +1, \mu_P\cdot l;l\in \N ; \\ 0 & \text{otherwise},  \end{cases}
$$
where $\mu_P=\begin{cases}
	n-d & \text{if $d$ is odd}, \\
	2(n-d) & \text{if $d$ is even}.
\end{cases}$
\end{Cor}
We again obtain the following results by Theorem \ref{ThmA} and \ref{ThmB}.
\begin{Thm}\label{main thm for comp2}
Let $\tau : \widehat W_d \rightarrow \widehat W_d$ be a fibered twist. Under the assumption in Corollary \ref{comp_exam2cor}, we have that
\begin{itemize}
	\item If a compactly supported symplectomorphism $\phi$ satisfies $[\phi] =  [\tau^l] \in \pi_0(\symp^c(\widehat W_d))$ for some $l \neq 0$, then $s_n(\phi) \geq 1$.
	\item a class $[\tau]$ has an infinite order in $\pi_0(\symp^c(\widehat{W}_d))$.
\end{itemize} 
\end{Thm}

\begin{Rem}
	In \cite[Example 7.14]{CDvK}, it is proven that $[\tau]$ has infinite order in $\pi_0(\symp^c(\widehat{W}_d))$ for $d>1$.
\end{Rem}

%\begin{Lem}
%		The real Lagrangian $L_{Q_k}$ in $Q_k$ is diffeomorphic to $\R P^{n-1}$ for odd $k$ and is diffeomorphic to $S^{n-1}$ for even $k$.
%		\end{Lem}
		
%		\begin{proof}
%Suppose that $k$ is odd. Note that 
%$
%		L_{Q_k} = \{[x_0:\cdots:x_n]|x_0^k+\cdots+x_n^k=0\}.
%$
%Define a diffeomorphism from $\mathbb{R} P^{n-1}$ to $L_{Q_k}$, $\Phi([x_0:\cdots:x_{n-1}])=[x_0:\cdots:x_{n-1}:-(x_0^k+\cdots +x_{n-1}^k)^{1/k}]$
%with the inverse map $\displaystyle \Phi^{-1}([x_0:\cdots:x_n])=[x_0:\cdots:x_{n-1}]$.
%Since at least one of $x_0,\cdots,x_{n-1}$ is not zero and $k$ is odd, we note that a function $-(x_0^k+\cdots+x_{n-1}^k)^{1/k}$ is uniquely well-defined and smooth.

%Suppose now that $k$ is even. Note that 
%$
%		L_{Q_k} = \{[x_0:\cdots:x_n]\ |\ x_0^k+\cdots+x_{n-1}^k=x_n^k,\quad x_n\ne 0\}.
%$
%Identifying $S^{n-1}$ with $\{(y_0,\cdots,y_{n-1})\in \R^n\ |\ y_0^k+\cdots+y_{n-1}^k=1\}$, we define a diffeomorphism from $S^{n-1}$ to $L_{Q_k}$, $\Psi(y_0,\cdots,y_{n-1})=[y_0:\cdots:y_{n-1}:1]$
%with the inverse map
%$\displaystyle \Psi^{-1}([x_0:\cdots:x_{n-1}:x_n])=\left(\frac{x_0}{x_n},\cdots,\frac{x_{n-1}}{x_n}\right)$.
%		\end{proof}

\end{document}